\documentclass[12pt]{amsart}
\usepackage{amssymb,amsmath,amsfonts,latexsym}
\usepackage{bm}
\usepackage[all,cmtip]{xy}
\usepackage{amscd}
\usepackage{setspace}
\usepackage{mathrsfs}
\usepackage[colorlinks,citecolor=blue,urlcolor=blue,linkcolor=black]{hyperref}

\allowdisplaybreaks

\newcommand{\bea}{\begin{eqnarray}}
\newcommand{\eea}{\end{eqnarray}}

\newcommand{\clb}{\mathcal{B}}

\newcommand{\cld}{\mathcal{D}}
\newcommand{\cle}{\mathcal{E}}
\newcommand{\clf}{\mathcal{F}}

\newcommand{\clh}{\mathcal{H}}
\newcommand{\clk}{\mathcal{K}}

\newcommand{\clm}{\mathcal{M}}
\newcommand{\cln}{\mathcal{N}}

\newcommand{\clq}{\mathcal{Q}}
\newcommand{\clr}{\mathcal{R}}
\newcommand{\cls}{\mathcal{S}}

\newcommand{\clx}{\mathcal{X}}
\newcommand{\cly}{\mathcal{Y}}

\def\textmatrix#1&#2\\#3&#4\\{\bigl({#1 \atop #3}\ {#2 \atop #4}\bigr)}
\def\dispmatrix#1&#2\\#3&#4\\{\left({#1 \atop #3}\ {#2 \atop #4}\right)}
\newcommand{\be}{\begin{equation}}
\newcommand{\ee}{\end{equation}}
\newcommand{\ben}{\begin{eqnarray*}}
\newcommand{\een}{\end{eqnarray*}}

\newcommand{\NI}{\noindent}

\newcommand{\bi}{\begin{itemize}}
\newcommand{\ei}{\end{itemize}}

\theoremstyle{definition}

\theoremstyle{plain}

\newtheorem{thm}{Theorem}[section]

\newtheorem{lem}[thm]{Lemma}

\theoremstyle{definition}
\newtheorem{defn}[thm]{Definition}
\newtheorem{rem}[thm]{Remark}
\newtheorem{ex}[thm]{Example}

\newcommand{\C}{\mathbb{C}}

\newcommand{\D}{\mathbb{D}}

\newcommand{\T}{\mathbb{T}}
\newcommand{\E}{\mathcal{E}} 
\newcommand{\F}{\mathcal{F}} 
  
\newcommand{\h}{{H}^2(\mathbb{D})}

\newcommand{\he}{{H}^2_\mathcal{E}(\mathbb{D})}

\newcommand{\hbe}{H^\infty_{\mathcal {B(E)}}(\mathbb{D})}

\numberwithin{equation}{section}

\let\phi=\varphi

\begin{document}

\title[Power partial isometries]{Power partial isometries}
	
\author[Babbar]{Kritika Babbar}
\address{Indian Institute of Technology Roorkee, Department of Mathematics,
		Roorkee-247 667, Uttarakhand,  India}
\email{kritika@ma.iitr.ac.in, kritikababbariitr@gmail.com}

\author[Maji]{Amit Maji}
\address{Indian Institute of Technology Roorkee, Department of Mathematics,
		Roorkee-247 667, Uttarakhand,  India}
\email{amit.maji@ma.iitr.ac.in, amit.iitm07@gmail.com ({Corresponding author)}}

\subjclass[2010]{47A15, 47A45, 47A65, 47A20, 46E22, 46E40, 47B35, 47B32}


\keywords{Partial isometries, reducing subspaces, invariant subspaces, hyperinvariant subspaces, Hardy space, truncated shift, Toeplitz operator, Hankel operator}

\begin{abstract}
In this paper we obtain a complete characterization of reducing, invariant, and hyperinvariant subspaces for the completely non-unitary component of a power partial isometry. In particular, precise characterization of reducing, invariant, and hyperinvariant subspaces of a truncated shift operator has been achieved.
\end{abstract}
\maketitle

\tableofcontents

\section{Introduction}
Partial isometries form an important class of operators in the linear analysis. They have widespread applications in various fields like in the polar decomposition of operators, C$^*$- algebras, quantum physics,  and they form a fundamental object in the dimension theory of von~Neumann algebras; see the papers \cite{Brenken-Niu}, \cite{Pearcy} for a thorough presentation of these connections. 
Partial isometries, on the other hand, are more complicated to study than isometries, so we restrict ourselves to the power partial isometries: operators whose all positive powers are partial isometries. The primary reason to do so is that power partial isometries have more concrete structure due to their algebraic properties. The main objective of this paper is to fully characterize the reducing, invariant, and hyperinvariant subspaces for the completely non-unitary (c.n.u. in short) component of a power partial isometry.

The decomposition of power partial isometries was first studied by Halmos and Wallen in \cite{halmospowers}, where they proved that every power partial isometry is a direct sum of a unitary operator, a unilateral shift, a backward shift, and truncated shifts. This work has inspired a great deal of mathematicians and it becomes one of the active areas of current research that brings together several areas of mathematics. Noteworthy contributions to this field have been made by Guyker \cite{guyker}, Burdak \cite{burdak}, Catepill\'{a}n, and Szyma\'{n}ski \cite{catepillan}, Huef, Raeburn, and Tolich \cite{huef}, and others. The canonical decomposition of a power partial isometry, on one hand, gives a unique orthogonal sum of a unitary operator and a c.n.u.\ operator; and, on the other hand, Halmos-Wallen decomposition yields that the combining components of unilateral shift, backward shift, and truncated shifts form the c.n.u.\ component of a power partial isometry. Since the spectral theorem provides a complete structure and representation for normal operators, in particular, unitary operators, it suffices to work with the c.n.u.\ power partial isometry.

The first natural step to understand the structure of linear operators on Hilbert spaces is investigating its reducing and invariant subspaces. Another related problem is to find all the hyperinvariant subspaces of an operator. These problems have fascinated the attention of various mathematicians, and many results are known for a handful of specific operators but many operators still remain a mystery.
Dilation theory is one of the effective tools to study the structure of a contraction. One of the most far-reaching results is Sz.-Nagy's dilation result which states that every pure contraction on a Hilbert space can be dilated to a unilateral shift on a larger Hilbert space (see \cite{sznagy}). Another celebrated result is Sarason's commuting lifting theorem \cite{sarason} which says that an operator that commutes with a pure contraction can be lifted to an operator commuting with the shift operator without altering the norm. These results play crucial roles in this paper to find the complete and precise characterization of reducing, invariant, and hyperinvariant subspaces of a truncated shift operator which is a particular case of power partial isometry and then we utilize it for the case of an arbitrary c.n.u.\ power partial isometry.

The study of reducing subspaces of unilateral shift was initiated by Halmos in his seminal paper \cite{shifts}. It captivates the interests of many mathematicians and it stimulates the study of reducing subspaces of various classes of operators, including analytic Toeplitz operators, truncated Toeplitz operators which can be found in \cite{Nordgren},  \cite{timotinreducingc00}, \cite{li2018reducibility}, \cite{reducibilitytto}.  The study of invariant subspaces has evolved significantly, but still it remains one of the most intriguing problems in operator theory to find a complete characterization of invariant subspaces of an operator. Beurling \cite{beurling} gave a striking result where he characterized all the invariant subspaces of shift operator on the Hardy space over the open unit disc which was later generalized by Lax and Halmos for vector-valued Hardy space (see \cite{lax},\cite{shifts}). Later on, many mathematicians proved some remarkable results about the invariant subspaces (see \cite{helson}, \cite{sarkar-inv1}, \cite{sarkar-book}, \cite{CHL-Beurling–Lax–Halmos}, \cite{CR-dual}). In \cite{timotin}, Timotin characterized the invariant subspaces of direct sum of unilateral shift and backward shift on scalar-valued Hardy space which was further generalized by Gu and Luo for vector-valued Hardy space in \cite{guinvariant}. Note that the direct sum of unilateral shift and backward shift is a particular case of a c.n.u.\ power partial isometry. In Section 3 of this note, we explicitly find the reducing subspaces of a c.n.u.\ power partial isometry. In Section 4, we obtain the complete structure of invariant subspaces of a c.n.u.\ power partial isometry.

The motivation behind hyperinvariant subspaces of an operator is that they give  information about the commutant of an operator. The notion of hyperinvariant subspaces was first introduced by Douglas and Pearcy in \cite{douglastopology}. They characterized the hyperinvariant subspaces of normal operators as well as of unilateral shift on a vector-valued Hardy space. Extending that, Douglas \cite{douglashyperinvariant} characterized hyperinvariant subspaces of a general isometry. Later on, hyperinvariant subspaces have been extensively studied by various mathematicians like  Wu, Uchiyama (see \cite{wuhyperinvariant},  \cite{uchiyamahyperinvariantc.0}, \cite{uchiyamahyperinvariantc0}). In \cite{uchiyamahyperinvariantc.0}, Uchiyama found the hyperinvariant subspaces of pure contractions. In Section 5, we present an explicit structure for all  the hyperinvariant subspaces of c.n.u.\ power partial isometry and in particular, a pure power partial isometry.
\vspace{.3 cm}

\section{Preliminaries}
In what follows, $\mathcal{H}$ denotes a separable Hilbert space over the complex field $\C$, $I_{\clh}$ as the identity operator on $\mathcal{H}$. We simply write $I$ if the Hilbert space is clear from the context. Let $\mathcal{B(H)}$ stand for the algebra of all bounded linear operators on $\mathcal{H}$. For $T \in \clb(\clh)$, we denote $\cln(T)$ and $\clr(T)$ as the kernel and range of $T$, respectively.  A closed subspace $\mathcal{M}$ is said to be invariant under $T\in \mathcal{B(H)}$ if $T\mathcal{M} \subseteq \clm$ and $\mathcal{M}$ reduces $T$ if $\mathcal{M}$ is invariant under $T$ and $T^*$ both. A closed subspace $\clm$ is hyperinvariant for $T$ if $\clm$ is invariant under all the operators commuting with $T$. We say $T$ is a contraction if $\Vert T h\Vert\leq \Vert h \Vert$ for all $h \in \mathcal{H}$ and it is a pure contraction if $T^{*m} \rightarrow 0$ as $m \rightarrow \infty$ in the strong operator topology. A contraction $T$ on $\mathcal{H}$ is said to be completely non-unitary (c.n.u.\ in short) if there does not exist any nonzero reducing subspace $\mathcal{L}$ of $\clh$ such that $T|_{\mathcal{L}}$ is unitary.  An operator $T \in\mathcal{B(H)}$ is said to be an isometry if $T^*T=I_{\clh}$, that is, $\Vert Th \Vert=\Vert h\Vert $ for all $h \in \mathcal{H}$. A partial isometry $T$ is a contraction such that $\Vert Th\Vert=\Vert h \Vert$ for all $h \in \cln (T)^\perp$.

Let $\cle$ and $\clf$ be two complex separable Hilbert spaces. $L^2_\cle$ denotes the space of $\cle$-valued square-integrable functions on the unit circle $\T$ with respect to the normalized Lebesgue measure. Let $H^2_{\cle}(\D)$ be the Hardy space of $\cle$-valued analytic functions $f$ on the open unit disc $\D = \{z \in \C : |z| <1 \}$ such that 
\[
\| f \| =\sup_{0 \leq r <1}\left(\frac{1}{2\pi}\int_0^{2\pi} \| f(re^{it}) \|_\cle^2\, dt\right) <\infty,
\]
and we often use $H^2_{\cle}(\D)$ as a closed subspace of $L^2_\cle$ (in the sense of radial limits). Let $L^\infty_{\clb(\cle,\clf)}$ denote the algebra of  operator-valued bounded functions on $\T$ and $H^\infty_{\clb(\cle,\clf)}$ ($H^\infty$ for the scalar case) stand for the algebra of operator-valued bounded analytic functions on $\D$. For $\Phi \in L^\infty_{\clb(\cle,\clf)}$, the Laurent operator (or multiplication operator) from $L^2_\cle$ to $L^2_\clf$, denoted by $L_\Phi$, is defined as
\[
L_\Phi h=\Phi h \quad (h \in H^2_\cle(\D)).
\]
 Let $P_+^{\cle}$ be the orthogonal projection of $L^2_\cle$ onto $H^2_\cle(\D)$. The Toeplitz operator $T_\Phi: H^2_\cle(\D)\rightarrow H^2_\clf(\D)$ is defined by 
\[
T_\Phi h=P_+^{\clf}(\Phi h) \quad (h \in H^2_\cle(\D)).
\]
In particular, if $\cle=\clf$ and $\Phi(z) =zI$, then we write $T_\Phi$ as $M_z^\cle$ which is the multiplication operator by the coordinate function $z$ on $H^2_\cle(\D)$. We will simply write $M_z$ if the underlying  Hilbert space $\cle$ is apparent from the context. The Toeplitz  operator $T_\Phi$ is characterized by the operator equation $(M_z^\clf)^*T_\Phi M_z^\cle=T_\Phi$. If $\Phi \in H^\infty_{\clb(\cle,\clf)}$, then $T_\Phi$ is called an analytic Toeplitz operator and is characterized by the equation $M_z^\clf T_\Phi=T_\Phi M_z^\cle$. Let $J$ on $L^2_\cle$ be defined by 
\[
(Jf)(e^{it})= f({e^{-it}}) \quad (f \in L^2_\cle, e^{it} \in \T).
\]
Then $J$ is a self-adjoint unitary map. For $\Phi \in L^\infty _{\clb(\cle,\clf)}$, the Hankel operator $H_\Phi: H^2_\cle(\D) \rightarrow H^2_\clf(\D)$ is defined by 
\[
H_\Phi h=P_+^{\clf}J(\Phi h) \quad (h \in H^2_\cle(\D)).
\]
The Hankel operator $H_\Phi$ is characterized by the equation $H_\Phi M_z^\cle=(M_z^\clf)^*H_\Phi$. The adjoint of $H_\Phi$ is given by $H_\Phi^*=H_{\tilde{\Phi}}$ where $\tilde{\Phi}(z)=\Phi(\bar{z})^* $. Also, recall $\Theta \in H^\infty_{\clb(\cle,\clf)}$ is inner if $\Theta(e^{it})^*\Theta(e^{it})=I_{\cle}$ a.e.\ on $\T$. Moreover, $\Theta$ is right extremal if $\Theta$ cannot be factored as $\Theta_1\Theta_2$, where $\Theta_1$ is inner and $\Theta_2$ is a non-constant inner function. The celebrated Beurling-Lax-Halmos theorem says that any invariant subspace of $M_z^\cle$ is of the form $\Theta H^2_{\cle_*}(\D)$, where $\Theta \in H^\infty_{\clb(\cle_*,\cle)}$ is an inner function and $\cle_*$ is a closed subspace of $\cle$ (see \cite{sznagy}, \cite{Peller-Book} for more details).

Recall the basic definitions and properties of the main objects of this paper, namely, truncated shift and power partial isometry. Let $k \geq 1$ be any natural number. Let $\clh_0$ be a Hilbert space and 
$\mathcal{H}=\underbrace{\clh_0 \oplus\clh_0 \oplus \cdots \oplus\clh_0}_{k\,\, times}$ which is identified with $\clh_0 \otimes \C^k$. A truncated shift of index $k$, denoted by $J_k$, is defined on $\clh$ as $J_k=0$ for $k=1$ and for $k \geq 2$,
\[
J_k(x_1,x_2, \ldots, x_k)=(0, x_1, \ldots, x_{k-1}) \quad (x_i \in \clh_0, i =1, \ldots, k).
\]
Thus the matrix representation of $J_k$ is of the form
\[
    J_k=\begin{pmatrix}
	0 & 0& 0 &\ldots &0 &0\\
	I_{\clh_0} & 0& 0 &\ldots &0 &0\\
	0 & I_{\clh_0}& 0 &\ldots &0 &0\\
	\vdots& \vdots &\vdots &\ldots &\vdots &\vdots\\
        0 & 0 & 0 &\ldots & 0 & 0\\
	0 & 0 &0 &\ldots &I_{\clh_0} &0
	\end{pmatrix}_{k\times k}.
\]

An operator $T \in \mathcal{B(H)}$ is said to be a power partial isometry if $T^n$ is a partial isometry for all $n \geq 1$. Isometries, co-isometries, and truncated shifts are some of the examples of power partial isometries. We record the following characterizations of a partial isometry and properties of power partial isometry which will be used frequently in the sequel.
\begin{lem}[cf. \cite{burdak}] 
Let $T \in \mathcal{B(H)}$. Then the following are equivalent:
\begin{enumerate}
\item $T$ is a partial isometry. 
\item $T^*$ is a partial isometry.
\item $T^*TT^*=T^*$.
\item $TT^*T=T$.
\item $TT^*$ is an orthogonal projection on $\clr(T)$.
\item $T^*T$ is an orthogonal projection on $\clr(T^*)$.
\end{enumerate}
\end{lem}
For every power partial isometry $T$, we write $E_k=T^{*k}T^k$ and $F_k=T^kT^{*k}$ for all $k \geq 0$ as the initial and final projections, respectively. Clearly, $E_k \geq E_{k+1}$ and $F_k  \geq F_{k+1}$ for all $k \geq 0$. 
\begin{lem} [cf. \cite{halmospowers}] Let $T \in \mathcal{B(H)}$ be a power partial isometry. Then 
\begin{enumerate}
\item $E_kE_l=E_lE_k$ and  $F_kF_l=F_lF_k$  for all $k,l \geq 0$.
\item $E_kF_l=F_lE_k$ for all $k,l \geq 0$.
\item $E_kF_l$ and $E_kE_l$ are orthogonal  projections.
\item $TE_{k+1}=E_{k}T$ for all $k \geq 0$.
\item $TF_{k}=F_{k+1}T$ for all $k \geq 0$.
\end{enumerate}
\end{lem}

We have the following decomposition theorem for a power partial isometry given by Halmos and Wallen in \cite{halmospowers}.
\begin{thm}
Let $T \in \clb(\clh)$ be a power partial isometry. Then there exist subspaces 
$\clh_u, \clh_s, \clh_b$, and $\clh_k\,\, (k \geq 1)$ reducing $T$   and 
\[
	\clh= \clh_u \oplus\clh_b \oplus \clh_s \oplus \left(\bigoplus\limits _{k=1}	^\infty \clh_k\right),
\]
such that $T|_{\clh_u}$ is a unitary,  $T|_{\clh_b}$ is a backward shift, $T|_{\clh_s}$ is a unilateral shift and $T|_{\clh_k}$ is a truncated shift of index $k$. 
\end{thm}

In \cite{burdak}, Burdak explicitly found all the orthogonal decomposition subspaces for a power partial isometry $T$ as follows:
\begin{align*}
 \clh_u&=\left( \bigcap\limits_{n=0}^\infty T^n \clh \right) \cap \left( \bigcap\limits_{n=0}^\infty T^{*n}\clh\right), \\
  \clh_s&=\left( \bigcap\limits_{n=0}^\infty T^{*n}\clh\right) \cap \left(\bigoplus\limits_{n=0}^\infty T^n(\cln(T^*))\right),\\
  \clh_b&=\left( \bigcap\limits_{n=0}^\infty T^{n}\clh \right)\cap \left( \bigoplus\limits_{n=0}^\infty T^{*n}(\cln(T))\right),\\
  \clh_t= \left(\bigoplus\limits _{k=1}	^\infty \clh_k\right) &=\left( \bigoplus\limits_{n=0}^\infty T^{n}(\cln(T^*))\right) \cap \left( \bigoplus\limits_{n=0}^\infty T^{*n}(\cln(T))\right).
\end{align*}
More precisely, we determine a representation of $\clh_k$ from the Halmos and Wallen decomposition theorem. In fact,
\[
\clh_k= \bigoplus\limits_{n=1}^k\left(F_{n-1}-F_n\right)\left(E_{k-n}-E_{k-n+1}\right)\clh.
\]
Then it is easy to see that
\[
\clh_k=\bigoplus\limits_{n=1}^k \clr\left(F_{n-1}-F_n\right) \cap \clr \left(E_{k-n}-E_{k-n+1}\right).
\]
Note that for any $p \geq 1$,
\begin{align*}
F_{p-1}-F_{p}&=F_{p-1}-F_{p}F_{p-1}\\
&=(I_{\clh}-F_p)F_{p-1}\\
&=P_{\cln(T^{*p})}P_{\clr(T^{p-1})}.
\end{align*}
Since $P_{\cln(T^{*p})}$ and $P_{\clr(T^{p-1})}$ commute, it follows that
$\clr(F_{p-1}-F_{p})=\cln(T^{*p})\cap \clr(T^{p-1})$. Similarly, $\clr(E_{p-1}-E_{p})=\cln(T^{p})\cap \clr(T^{*(p-1)})$.  We now prove that
\[
\cln(T^{*p}) \cap \clr(T^{p-1})=T^{p-1}(\cln(T^*)).
\]
Let $x \in \cln(T^{*p}) \cap \clr(T^{p-1})$. Then $T^{p-1}T^{*(p-1)}x=x$ and $T^*(T^{*(p-1)}x)= T^{*p}x=0$. This implies $x \in T^{p-1}(\cln(T^*))$. Hence 
\[
\cln(T^{*p}) \cap \clr(T^{p-1}) \subseteq T^{p-1}(\cln(T^*)).
\] 

Suppose that $y \in \cln(T^*)$. Then 
\[
T^{*p}T^{p-1}y=T^*T^{*(p-1)}T^{p-1}y=T^{*p}T^pT^*y=0,
\]
as $T^*E_{p-1} = E_pT^*$.
Thus $T^{p-1}y \in \cln(T^{*p}) \cap \clr(T^{p-1})$. Hence $\cln(T^{*p}) \cap \clr(T^{p-1})=T^{p-1}(\cln(T^*))$. Similarly, $\cln(T^{p}) \cap \clr(T^{*(p-1)})=T^{*(p-1)}(\cln(T))$. Therefore, we get
\[
\clh_k=\bigoplus\limits_{n=1}^k T^{n-1}(\cln(T^*)) \cap T^{*(k-n)}(\cln(T)).
\]

Returning back to the well-known \textit{Sz.-Nagy and Foia\c{s}'} dilation and the 
Sarason's commuting lifting theorem which are essential tools in comprehending the structure of contractions and hence for any operator.

Let $T$ be a contraction on $\clh$. We say $M_z$ on $\he$ is an isometric dilation of $T$ on $\clh$ if there exists an isometry $\Pi: \clh \rightarrow \he$ (cf. \cite{sarkar-book}) such that
\[
\Pi T^*=M_z^*\Pi.
\]
Suppose $T$ is a pure contraction on $\clh$, $D_{T^*}=(I_{\clh} - TT^*)^{1 \over 2} $, and $\cld_{T^*} = \overline{\clr(D_{T^*})}$. Define an isometry $\Pi: \clh \rightarrow H^2_{\cld_{T^*}}(\D)$ by
\[
(\Pi h)(z)=D_{T^*}(I-zT^*)^{-1}h \quad (z \in \D, h \in \clh).
\]
Then it is easy to see that $\Pi T^*=M_z^*\Pi.$ Hence $M_z$ is an isometric dilation of $T$ and $\Pi \clh$ is $M_z^*$ invariant subspace of $H^2_{\cld_{T^*}}(\D)$. Thus $ T$ is unitarily equivalent to $P_{\Pi \clh} M_z|_{\Pi \clh}$, that is, $T \cong P_{\Pi \clh} M_z|_{\Pi \clh}$, and we often use $T=P_{\Pi \clh} M_z|_{\Pi \clh}$  if it is clear from the context. 
Also note that 
\[
H^2_{\cld_{T^*}}(\D)=\bigvee\{z^m\Pi h: m \in \mathbb{Z}_+, h \in \clh\},
\]
which means that $M_z$ is the minimal isometric dilation of $T$.
In the sequel, due to Halmos and Wallen decomposition, we work with the c.n.u.\ power partial isometry 
\[
T=M_z^\cle \oplus (M_z^\clf)^* \oplus \bigoplus\limits_{k=1}^\infty J_k
\]
on $H^2_\cle(\D) \oplus H^2_{\clf}(\D) \oplus \bigoplus \limits_{k=1}^\infty  \mathbb{E}_k $, where $J_k$ is the truncated shift of index $k$ on 
\[
\mathbb{E}_k =\underbrace{\cle_k \oplus \cle_k\oplus \ldots \oplus \cle_k}_{k\,\,times} =\cle_k \otimes \C^k,
\]
and $\cle,\clf,\text{ and } \cle_k$ are complex separable Hilbert spaces for all $k \geq 1$ (cf. \cite[Theorem 2.1]{huef}).

\begin{thm}[cf. \cite{FF-commutant}]
Let $T_1$ and $T_2$ be contractions on $\cle_1$ and $\cle_2$ and let $V_1$ and $V_2$ be their minimal isometric dilations on $\clf_1$ and $\clf_2$, respectively. For every bounded operator $X \in \clb(\cle_2,\cle_1)$ satisfying $T_1X=XT_2$, there exists $Y \in \clb(\clf_2,\clf_1)$ such that
\begin{enumerate}
\item $V_1 Y=Y V_2$,
\item $X=P_{\cle_1}Y |_{\cle_2}$,
\item $Y(\clf_2 \ominus \cle_2) \subseteq \clf_1 \ominus \cle_1$, and
\item $\Vert X \Vert=\Vert Y\Vert.$
\end{enumerate}
\end{thm}

We conclude this section by reviewing the notion of range functions (see \cite{helson}) which will be helpful in finding the invariant subspaces of power partial isometry.

\begin{defn}
A range function $K$ of $\mathcal{E}$ is a function on the unit circle $\T$ which takes values in the closed subspaces of $\cle$. $K$ is said to be measurable if the orthogonal projection $P(e^{it})$ on $K(e^{it})$ is measurable, that is, $\langle P(e^{it})a_1,a_2 \rangle$ is a measurable scalar function for all $a_1, a_2 \in \cle$.  
\end{defn}

For each measurable function $K$ of $\cle$, we denote $\clm_{K}$ as the set of all functions $f \in L^2_\cle$ such that $f(e^{it}) \in K(e^{it})$ a.e.\ on $\T$.
\vspace{.3 cm}

\section{Reducing subspaces of a c.n.u.\ power partial isometry}

In this section, we discuss reducing subspaces of truncated shift of index $k$. Using this result, we obtain the reducing subspaces of c.n.u.\ power partial isometry.

\begin{thm}\label{thm-reducingtruncatedshift}
Let $\cle$ be a Hilbert space and $k$ be any natural number. Then a closed subspace $\clm$ of $\cle \otimes \mathbb{C}^k$ is reducing for truncated shift of index $k$ on $\cle \otimes \mathbb{C}^k$ if and only if  $\clm=\cls \otimes \mathbb{C}^k$, where $\cls$ is a closed subspace of $\cle$.
\end{thm}

\begin{proof}
Suppose $J_k$ is truncated shift operator of index $k$ on $\mathcal{E}\otimes \mathbb{C}^k$.
For $k=1$, $J_k=0$ and thus the result holds trivially. Assume that $k \geq 2$. Define a map $\Pi:\mathcal{E}\otimes \mathbb{C}^k \rightarrow H^2_\mathcal{E}(\D)$ by
\[
\Pi(a_0,a_1,\ldots,a_{k-1})=\sum\limits_{m=0}^{k-1} a_mz^m \qquad (a_i \in \cle).
\]
Clearly, $\Pi$ is an isometry such that $\Pi J_k^*=M_z^*\Pi$. More precisely, $M_z$ is the minimal isometric dilation of $J_k$ such that 
\[
J_k \cong P_{\clq}M_z|_{\clq},
\]
where $\clq=\Pi(\mathcal{E}\otimes \mathbb{C}^k)= [z^k\he]^\perp$.

Let $\clm$ be a reducing subspace of $J_k$ and $P_\clm$ be the orthogonal projection of $\cle \otimes \mathbb{C}^k$ onto $\clm$. Then $P_\clm J_k=J_kP_\clm$ which implies  $P_\clm(\cln(J_k^*)) \subseteq \cln(J_k^*)$. Also note that 
\[
\cln(J_k^*)= \cle \oplus \underbrace{\{0\} \oplus \cdots \oplus \{0\}}_{(k-1)\,\,times}.
\]
Hence, for each $a \in \cle$ (considered as a constant function in $H^2_\cle(\D)$), $\Pi P_\clm \Pi^* a \in \cle$. Now define a map $P: H^2_\mathcal{E}(\D)\rightarrow H^2_\mathcal{E}(\D)$ by 
\[
P\left(\sum\limits_{m=0}^\infty a_mz^m\right):=\Pi P_\clm \Pi^*\left(\sum\limits_{m=0}^{k-1}a_mz^m\right)+\sum\limits_{m=k}^\infty (\Pi P_\clm \Pi^*a_m)z^{m}.
\]
It is easy to check that $P$ is a well defined map on $H^2_\mathcal{E}(\D)$ and $P\Pi=\Pi P_\clm $. We first prove that $P$ is a projection. Let $f=\sum\limits_{m=0}^\infty a_m z^m, g= \sum\limits_{m=0}^\infty b_mz^m \in H^2_\cle(\D)$. Suppose
\[
\Pi P_\clm\Pi^*a_m =x_m \in \cle \mbox{~~for each~~} m.
\]
Consider
\begin{align*}
P^2 f = P^2\left(\sum\limits_{m=0}^\infty a_mz^m\right)&=P\left( \Pi P_\clm \Pi^*\left(\sum\limits_{m=0}^{k-1}a_mz^m\right)+\sum\limits_{m=k}^\infty (\Pi P_\clm \Pi^*a_m)z^{m}\right)\\
&=(\Pi P_\clm \Pi^* )^2\left(\sum\limits_{m=0}^{k-1}a_mz^m \right) +P\left(\sum\limits_{m=k}^\infty x_mz^m \right)\\
&=\Pi P_\clm \Pi^* \left(\sum\limits_{m=0}^{k-1}a_mz^m \right) + \sum\limits_{m=k}^\infty (\Pi P_\clm \Pi^*x_m)z^{m}\\
&=\Pi P_\clm \Pi^* \left(\sum\limits_{m=0}^{k-1}a_mz^m \right) + \sum\limits_{m=k}^\infty ((\Pi P_\clm \Pi^*)^2 a_m)z^m\\
&=\Pi P_\clm \Pi^* \left(\sum\limits_{m=0}^{k-1}a_mz^m \right) + \sum\limits_{m=k}^\infty (\Pi P_\clm \Pi^*a_m)z^m\\
&=P\left(\sum\limits_{m=0}^\infty a_mz^m\right) \\
&= Pf.
\end{align*} 

Also,
\begin{align*}
\left\langle f, P^*g \right\rangle& =\left \langle P\left(\sum\limits_{m=0}^\infty a_mz^m\right), \sum\limits_{m=0}^\infty b_mz^m\right \rangle \\
& =\left\langle \Pi P_\clm \Pi^*\left(\sum\limits_{m=0}^{k-1}a_mz^m\right)+\sum\limits_{m=k}^\infty (\Pi P_\clm \Pi^*a_m)z^m, \sum\limits_{m=0}^\infty b_mz^m \right\rangle\\
&= \left\langle \Pi P_\clm \Pi^*\left(\sum\limits_{m=0}^{k-1} a_mz^m\right), \sum\limits_{m=0}^{k-1}b_mz^m \right\rangle +\left \langle \sum\limits_{m=k}^\infty (\Pi P_\clm \Pi^*a_m)z^m, \sum\limits_{m=k}^\infty b_mz^m \right\rangle\\ 
&=  \left \langle \sum\limits_{m=0}^{k-1}a_mz^m, \Pi P_\clm \Pi^*\left(\sum\limits_{m=0}^{k-1} b_mz^m\right) \right\rangle + \sum\limits_{m=k}^\infty \left\langle \Pi P_\clm \Pi^*a_m , b_m \right\rangle\\
&=  \left \langle \sum\limits_{m=0}^{k-1}a_mz^m, \Pi P_\clm \Pi^*\left(\sum\limits_{m=0}^{k-1} b_mz^m\right) \right\rangle + \sum\limits_{m=k}^\infty \left\langle a_m , \Pi P_\clm \Pi^*b_m \right\rangle\\
&=  \left \langle \sum\limits_{m=0}^{k-1}a_mz^m, \Pi P_\clm \Pi^*\left(\sum\limits_{m=0}^{k-1} b_mz^m\right) \right\rangle + \left\langle \sum\limits_{m=k}^{\infty}a_mz^m , \sum\limits_{m=k}^{\infty} (\Pi P_\clm \Pi^*b_m)z^{m} \right\rangle\\
&=\left\langle f, \Pi P_\clm \Pi^*\left(\sum\limits_{m=0}^{k-1}b_mz^m\right)+\sum\limits_{m=k}^\infty (\Pi P_\clm \Pi^* b_m)z^m\right\rangle.
\end{align*}
Therefore, 
\begin{align*}
P^*g
& = \Pi P_\clm \Pi^*\left(\sum\limits_{m=0}^{k-1}b_mz^m\right)+\sum\limits_{m=k}^\infty (\Pi P_\clm \Pi^*b_m) z^{m}
=P\left(\sum\limits_{m=0}^\infty b_mz^m \right) =Pg.
\end{align*}
Hence $P$ is an orthogonal projection on $H^2_\cle(\D)$. 

We now prove that $P$ commutes with $M_z$. To do that, first recall that 
\begin{align*}
\Pi J_k^* & = M_z^*\Pi \\ \mbox{and~~~}
P_\clm J_k &= J_k P_\clm.
\end{align*}
Consider
\begin{align*}
PM_zf &=P\left(\sum \limits_{m=0}^\infty a_mz^{m+1}\right)\\
 &=\Pi P_\clm \Pi^*\left(\sum\limits_{m=0}^{k-2}a_{m}z^{m+1}\right)+\sum\limits_{m=k-1}^\infty (\Pi P_\clm \Pi^*a_{m})z^{m+1}\\
 &=\Pi P_\clm \Pi^*M_z\left(\sum\limits_{m=0}^{k-2}a_mz^m\right)+ \left(\Pi P_\clm \Pi^*a_{k-1}\right)z^k +\sum\limits_{m=k}^\infty  (\Pi P_\clm \Pi^*a_{m})z^{m+1}\\
&=\Pi P_\clm J_k\Pi^*\left(\sum\limits_{m=0}^{k-2}a_mz^m\right)+\Pi P_\clm(a_{k-1},0, \ldots,0)z^k+\sum\limits_{m=k}^\infty  (\Pi P_\clm \Pi^* a_{m})z^{m+1}\\
&= \Pi J_k P_\clm \Pi^* \left(\sum\limits_{m=0}^{k-2}a_mz^m\right)+\Pi P_\clm (J_k^*)^{k-1}(0, \ldots, 0, a_{k-1})z^k + \sum\limits_{m=k}^\infty  (\Pi P_\clm \Pi^* a_{m})z^{m+1}\\
&= \Pi J_k\Pi^*\Pi P_\clm \Pi^*\left(\sum\limits_{m=0}^{k-2}a_mz^m\right)+ \Pi (J_k^*)^{k-1} P_\clm (0, \ldots, 0, a_{k-1})z^k   +\sum\limits_{m=k}^\infty  (\Pi P_\clm \Pi^* a_{m})z^{m+1}\\
&=\Pi \Pi^* M_z \Pi  P_\clm \Pi^* \left(\sum\limits_{m=0}^{k-2}a_mz^m\right)+ (M_z^*)^{k-1} \Pi P_\clm \Pi^*(a_{k-1}z^{k-1})z^k   +\sum\limits_{m=k}^\infty  (\Pi P_\clm \Pi^*a_{m})z^{m+1}\\
&= M_z \Pi P_\clm \Pi^* \left(\sum\limits_{m=0}^{k-2}a_mz^m\right)+ \Pi P_\clm \Pi^*(a_{k-1}z^{k-1})z   +\sum\limits_{m=k}^\infty  (\Pi P_\clm \Pi^* a_{m})z^{m+1}\\
&=M_z\left( \Pi P_\clm \Pi^* \left(\sum\limits_{m=0}^{k-1}a_mz^m\right) +\sum\limits_{m=k}^\infty (\Pi P_\clm \Pi^*a_m)z^m \right)\\
&=M_zPf.
\end{align*}
The third last equality holds because $ M_z \Pi  P_\clm \Pi^* \left(\sum\limits_{m=0}^{k-2}a_mz^m\right) \in \clr(\Pi)$. Therefore, $PM_z=M_z P$ yields $P=T_\Phi$ for some $\Phi \in \hbe$. Since $P$ is a projection, $\Phi$ is a constant function, say $\phi_0$, on $\mathcal{E}$ such that $\phi_0=\phi_0^*=\phi_0^2$. Thus $P_\clm =\Pi^* T_{\phi_0}\Pi$ infers 
\begin{align*}
\clm =\Pi^*T_{\phi_0}\left([z^k\he]^\perp\right)& =\Pi^*T_{\phi_0}\left(\bigvee\limits_{m=0}^{k-1}\{z^m\eta: \eta \in \E\}\right) \\ 
& =\Pi^*\left(\bigvee\limits_{m=0}^{k-1}\{z^m\phi_0(\eta): \eta \in \E\}\right)\\
& =\underbrace{\phi_0(\E) \oplus \phi_0(\E) \oplus \cdots \oplus \phi_0(\E)}_{k\,\, times},
\end{align*}
where $\phi_0(\E)$ is a closed subspace of $\cle$. Hence any reducing subspace of $J_k$ is of the form $\cls \otimes \mathbb{C}^k$, where $\cls$ is a closed subspace of $\cle$. 

The converse is trivially true. The proof is completed.
\end{proof}

\begin{thm}
Let $J_k$ be the truncated shift of index $k$ on $\mathbb{E}_k =\cle_k \otimes \C^k$ for $k \geq 1$ and let $T= M_z^\cle \oplus (M_z^\clf)^* \oplus \bigoplus\limits_{k=1}^\infty J_k$ be a c.n.u.\ power partial isometry on $\mathcal{H}=H^2_\cle(\D) \oplus H^2_{\clf}(\D) \oplus \left(\bigoplus \limits_{k=1}^\infty \mathbb{E}_k\right)$. Then the only reducing subspaces of $T$ are those of the form $H^2_{\cle_*}(\D) \oplus H^2_{\clf_*}(\D) \oplus \left(\bigoplus\limits_{k=1}^\infty \cls_k \otimes \C^k\right)$, where $\cle_*, \clf_*$, and $\cls_k$ are closed subspaces of $\cle,\clf$, and $\cle_k$, respectively.
\end{thm}

\begin{proof}
Following Halmos-Wallen decomposition of the c.n.u.\ power partial isometry $T$ on $\clh$, we have $\mathcal{H}_s=H^2_\cle(\D), \clh_b=H^2_\clf(\D)$ and $\clh_k=\mathbb{E}_k$ for $k \geq 1$. 
Let $\clm$ be a reducing subspace of $T$.  We claim that $P_{\clh_i}\clm=\clh_i \cap \clm$ for $i=s,b,k \,\, \forall \,\,k\geq 1,$ and 
\[
\clm= P_{\clh_s}\clm \oplus P_{\clh_b}\clm \oplus \left(\bigoplus\limits_{k =1}^{\infty}P_{\clh_k} \clm\right)=(\clh_s\cap \clm) \oplus (\clh_b\cap \clm) \oplus \bigoplus\limits_{k = 1}^\infty(\clh_k\cap \clm).
\]
To prove the first claim, it suffices to show that $P_{\clh_i}\clm \subseteq \clm$
which is equivalent to prove that $P_{\clh_i}P_\clm =P_\clm P_{\clh_i}$, i.e., $P_\clm \clh_i \subseteq \clh_i$ for $i=s,b,k \,\, \forall \,\,k\geq 1$.

Let $x \in \clh_s= \bigcap\limits_{n=0}^\infty T^{*n}\clh \cap \bigoplus\limits_{n=0}^\infty T^n(\cln(T^*))$. Then, for $n \geq 0$, $x=T^{*n}x_n$ for some $x_n \in \clh$ and $x=\sum\limits_{n=0}^\infty T^n y_n$, where $y_n \in \cln(T^*)$. Now 
\[
P_\clm x=P_\clm T^{*n}x_n=T^{*n}P_\clm x_n \in T^{*n}\clh,
\]
where the second equality follows as $\clm$ reduces $T$ and hence $T^{*n}$. 
Also, 
\[
P_\clm x=P_\clm \left(\sum\limits_{n=0}^\infty T^ny_n\right)=\sum\limits_{n=0}^\infty  T^nP_\clm y_n\in  \bigoplus\limits_{n=0}^\infty T^n(\cln(T^*))
\]
as $T^*P_\clm y_n=P_\clm T^*y_n=0.$ Therefore, $P_\clm \clh_s \subseteq \clh_s$.

Interchanging the roles of $T$ and $T^*$ yields $P_\clm \clh_b \subseteq \clh_b$. By a similar argument, we can prove that $P_\clm \clh_k \subseteq \clh_k$ for all $k\geq 1$.

For the next claim, clearly $\clm \subseteq  P_{\clh_s} \clm \oplus P_{\clh_b}\clm \oplus \left(\bigoplus\limits_{k =1}^\infty P_{\clh_k}\clm \right)$. And since for all $i=b,s,k \geq 1$,  $P_{\clh_i}\clm =\clh_i \cap \clm \subseteq \clm$. Hence
\[
\clm = P_{\clh_s}\clm \oplus P_{\clh_b}\clm \oplus \left(\bigoplus\limits_{k =1}^{\infty}P_{\clh_k} \clm\right).
\]
Also, since $\clm$ reduces $T$, $\clh_i \cap \clm$ reduces $T|_{\clh_i}$ for each $i$. Hence any reducing subspace $\clm$ of $T$ is of type $\clm_s \oplus \clm_b \oplus \left(\bigoplus\limits_{k=1}^\infty \clm_k\right)$, where $\clm_i \subseteq \clh_i$ reduces $T|_{\clh_i}$ for $i=b,s,k \geq 1$.

As we know any reducing subspace of $M_z^{\cle}$ on $H^2_{\cle}(\D)$ is of the form $H^2_{\E_*}(\D)$, where $\E_*$ is a closed subspace of $\cle$ and with the help of Theorem \ref{thm-reducingtruncatedshift}, any reducing subspace of $T$  is of the form 
\[
\clm=H^2_{\E_*}(\D) \oplus H^2_{\F_*}(\D) \oplus \left(\bigoplus\limits_{k=1}^\infty \cls_k \otimes \mathbb{C}^k\right),
\]
where $\cle_*, \clf_*$, and $\cls_k$ are closed subspaces of $\cle,\clf$, and 
$\cle_k$, respectively.

Converse is easy to prove.
\end{proof}

\begin{rem} 
The above proof yields that reducing subspaces of a c.n.u.\ power partial isometry are always splitting subspaces. However, this does not hold in general. In fact, in the case of the identity operator $I_{\clh \oplus \clk}$ on $\clh \oplus \clk$, every subspace reduces $I_{\clh \oplus \clk}$ but not every subspace is splitting subspace.  
\end{rem}
\vspace{.3 cm}

\section{Invariant subspaces of a c.n.u.\ power partial isometry}

In this section we obtain the complete structure of invariant subspaces of the c.n.u.\ power partial isometry $T= (M_z^\clf)^* \oplus M_z^\cle\oplus \left( \bigoplus\limits_{k=1}^\infty J_k\right)$ on $H^2_\clf(\D) \oplus H^2_\cle(\D) \oplus \left(\bigoplus\limits_{k=1}^\infty \mathbb{E}_k\right) $, where  
$\mathbb{E}_k =\cle_k \otimes \C^k$. In order to find the invariant subspaces of $T$, first we make some reductions which are inspired by the work done by  Gu and Luo in \cite{guinvariant}.

We first obtain the invariant subspaces of $T^*=  M_z^\clf \oplus (M_z^\cle)^* \oplus \left(\bigoplus \limits_{k=1}^\infty J_k^*\right)$. To do that, first set $\cle'=\bigoplus\limits_{k=1}^\infty \cle_k$ and $\mathbb{E}= \bigoplus\limits_{k=1}^\infty \mathbb{E}_k$. Also $\mathbb{E}_k$ is identified with $[z^k H^2_{\cle_k}(\D)]^{\perp}$. Now onwards, we just write $\mathbb{E}_k =  [z^k H^2_{\cle_k}(\D)]^{\perp}$ for $k \geq 1$. Since $M_z^{\cle'}$ on $H^2_{\cle'}(\D)$ is the minimal isometric dilation of $\bigoplus\limits_{k=1}^\infty J_k$, thus we simply write
$\bigoplus\limits_{k=1}^\infty J_k =P_{\mathbb{E}}M_z^{\cle'}|_\mathbb{E}$, and 
$\bigoplus\limits_{k=1}^\infty J_k^*=(M_z^{\cle'})^*|_{\mathbb{E} }$, where $J_1 =0$ and
\begin{equation}\label{def}
J_k\left( \sum_{m=0}^{k-1} a_m z^m\right) = \sum_{m=0}^{k-2} a_m z^{m+1} \quad (k \geq 2).
\end{equation}

We make the following reductions to find the invariant subspaces of $T^*$. Let $\cln$ be an invariant subspace of $T^*$. Consider
\[
T_1=\begin{pmatrix}
M_z& 0\\
0 & M_z^*
\end{pmatrix}:H^2_{\clf}(\D)\oplus H^2_{\cle \oplus \cle'}(\D) \rightarrow H^2_{\clf}(\D)\oplus H^2_{\cle \oplus\cle'}(\D).
\] 
$\cln$ is invariant under $T^*$ implies $\cln$ is also invariant under $T_1$. Let
\[
T_2=\begin{pmatrix}
L_z & 0\\
0& M_z^*
\end{pmatrix}: L^2_{\clf} \oplus H^2_{\cle \oplus \cle'}(\D) \rightarrow L^2_{\clf} \oplus H^2_{\cle \oplus \cle'}(\D).
\]
Now $\cln\subseteq H^2_\clf(\D) \oplus H^2_\cle(\D) \oplus \mathbb{E}$ is $T_1$-invariant infers that $\cln$ is $T_2$-invariant.
Again let
\[
T_3=\begin{pmatrix}
L_z & 0\\
0 & M_z 
\end{pmatrix}:  L^2_{\clf} \oplus H^2_{\cle \oplus \cle'}(\D) \rightarrow L^2_{\clf} \oplus H^2_{\cle \oplus \cle'}(\D).
\]
Recall the self-adjoint unitary map $J: L^2_{\clf} \rightarrow L^2_{\clf}$ defined as
\[
Jf(e^{it})=f(e^{-it}) \quad \quad ( f \in L^2_{\clf}, e^{it} \in \T).
\]
Note that $L_z J=JL_z^*$ on $L^2_\clf$. Define
\[
J'=\begin{pmatrix}
J & 0\\
0& I
\end{pmatrix}:  L^2_{\clf} \oplus H^2_{\cle \oplus \cle'}(\D) \rightarrow L^2_{\clf} \oplus H^2_{\cle \oplus \cle'}(\D).
\]
Clearly, $\cln$ is invariant under $T_2$ if and only if $\cln':=J'\cln$ is invariant under $T_3^*$. Hence, the problem reduces to finding the invariant subspace $\cln'$ of $T_3^*$ such that 
\[
	\cln=J'\cln' \subseteq H^2_\clf(\D) \oplus H^2_\cle(\D) \oplus \mathbb{E}.
\]
or
\[
	\cln'\subseteq \clk \oplus H^2_\cle(\D) \oplus \mathbb{E}, \text{ where } \clk=\bigvee\{e^{imt}\eta:\eta \in \clf, m \leq 0\}.
\]
It is equivalent of saying that  
$\cln'':=\left( L^2_{\clf} \oplus H^2_{\cle \oplus \cle'}(\D)\right) \ominus \cln'$ is invariant subspace of $T_3$ such that $zH^2_{\clf}(\D) \oplus \left(\bigoplus\limits_{k=1}^\infty z^kH^2_{\cle_k}(\D)\right) \subseteq \cln''$. Consider
$$T_4=\begin{pmatrix}
L_z & 0\\
0 & L_z
\end{pmatrix}:  L^2_{\clf} \oplus L^2_{\cle \oplus \cle'} \rightarrow  L^2_{\clf} \oplus L^2_{\cle \oplus \cle'}.$$
Thus $\cln''$ is invariant under $T_3$ infers that $\cln''$ is also invariant under $T_4$ such that $\cln'' \subseteq  L^2_{\clf} \oplus H^2_{\cle \oplus \cle'}(\D)$. By Helson's results in \cite{helson}, the invariant subspaces of the bilateral shift $T_4$ are precisely those of the form 
\[
\cln''= \Phi H^2_{\clf_0}(\D) \oplus \clm_K,
\]
where $\clf_0 \subseteq \clf\oplus\cle \oplus \cle'$ , $\Phi \in L^\infty_{\clb(\clf_0,\clf\oplus \cle \oplus \cle')}$ is zero or isometric-valued function and $K$ is a measurable range function on $\clf$ such that $K(e^{it}) \perp \Phi(e^{it})\clf_0$ a.e.\ on $ \T$. Note that $\clm_K$ is $L_z$-reducing.
Let $P_\clf$, $P_\cle$, and $P_{\cle'}$ be the orthogonal projections from $\clf \oplus \cle \oplus \cle'$ onto $\clf,\cle$, and $\cle'$, respectively.
Set 
\[
\Phi(e^{it})=\begin{pmatrix}
\Phi_\clf({e^{it}})\\
\Phi_\cle(e^{it})\\
\Phi_{\cle'}(e^{it})
\end{pmatrix},
\]
where $\Phi_\clf(e^{it}) =P_\clf \Phi(e^{it})$, $\Phi_\cle(e^{it})=P_\cle\Phi(e^{it})$, and $\Phi_{\cle'}(e^{it})=P_{\cle'}\Phi(e^{it})$. Then $\Phi_\clf  \in L^\infty_{\clb(\clf_0,\clf)}$, $\Phi_\cle \in  L^\infty_{\clb(\clf_0,\cle)}$, and $\Phi_{\cle'} \in   L^\infty_{\clb(\clf_0,\cle')}$.

As proved  in \cite[Lemma 3.5]{guinvariant}, $\cln'' \subseteq  L^2_{\clf} \oplus H^2_{\cle \oplus \cle'}(\D)$ if and only if   $\clm_K \subseteq L_\clf^2$, $\Phi_{\cle} \in H^\infty_{\clb(\clf_0,\cle)}$, and  $\Phi_{\cle'} \in H^\infty_{\clb(\clf_0, \cle')}$. 

Note that $\cln= \left(L^2_\clf\oplus H^2_{\cle\oplus\cle'}(\D)\right)\ominus J'\cln''$ and  $\cln\subseteq H^2_{\clf}(\D) \oplus H^2_\cle(\D) \oplus \mathbb{E}$. Let $f \in H^2_{\clf}(\D)$, $g \in H^2_{\cle}(\D)$, and $h \in \mathbb{E}$. Then $f \oplus g \oplus h \perp J'\cln''$ if and only if 
\begin{enumerate}
\item $Jf \perp \clm_K$, and \label{1}
\item $Jf \oplus g \oplus h \perp \Phi H^2_{\clf_0}(\D)$. \label{2}
\end{enumerate}

For (\ref{1}), let 
\[
\clh_1=\{u \in H^2_{\clf}(\D): u \perp \clm_{K(e^{-it})}\}.
\]
By \cite[Lemma 4.7]{guinvariant}, $\clh_1$ is invariant under $M_z^\clf$ and when $\clh_1 \neq \{0\}$, there exist a Hilbert space $\clf_1 \subseteq \clf$ and an inner and  right extremal function $\Theta \in H^\infty_{\clb(\clf_1,\clf)}$ such that 
\[
\clh_1 =\Theta H^2_{\clf_1}(\D). 
\]

And for (\ref{2}), $Jf \oplus g \oplus h \perp \Phi H^2_{\clf_0}(\D)$ if and only if  $\Phi^*(Jf \oplus g \oplus h) \perp H^2_{\clf_0}(\D)$.
Again
\begin{align*}
\Phi^*(Jf\oplus g\oplus h) \perp H^2_{\clf_0}(\D) \Leftrightarrow& P_+^{\clf_0}\Phi^*(Jf \oplus g\oplus h)=0\\
	\Leftrightarrow & P_+^{\clf_0}(\Phi_\clf^*Jf+\Phi_\cle^*g+ \Phi_{\cle'}^* h)=0\\
	\Leftrightarrow & H_{\Phi_\clf}^*f+ T_{\Phi_\cle}^*g + T_{\Phi_{\cle'}}^*h=0\\
	\Leftrightarrow&  f \oplus  g \oplus h \in 
	\cln \left(\begin{bmatrix}
H_{\Phi_\clf}^* & T_{\Phi_\cle}^* & T_{\Phi_{\cle'}}^*\big|_{\mathbb{E}}
\end{bmatrix}\right).
\end{align*}
Hence 
\[
\cln= \cln(W_\Phi) \cap \left(\Theta H^2_{\clf_1}(\D) \oplus H^2_{\cle}(\D) \oplus \mathbb{E}\right),
\]
where 
\[
W_\Phi=
\begin{bmatrix}
H_{\Phi_\clf}^* & T_{\Phi_\cle}^* & T_{\Phi_{\cle'}}^*\big|_{\mathbb{E}}
\end{bmatrix}: H^2_\clf(\D) \oplus H^2_{\cle}(\D) \oplus \mathbb{E} \rightarrow H^2_{\clf_0}(\D),
\]
and $\Theta \in H^\infty_{\clb(\clf_1,\clf)}$ is inner and right extremal function.
Hence
\[
\clm=\left(H^2_\clf(\D) \oplus H^2_\cle(\D) \oplus \mathbb{E}\right) \ominus\cln=\bigvee\left\{\clr(W_\Phi^*), \left(\cln(T_\Theta^*) \oplus \{0\} \oplus \{0\}\right)\right\}.
\]
The converse is also true. Indeed, 
\[
W_\Phi^*=\begin{bmatrix}
H_{\Phi_\clf} & T_{\Phi_\cle} &P_\mathbb{E} T_{\Phi_{\cle'}}
\end{bmatrix}^T.
\]
So we have
\[
	\begin{pmatrix}
	(M_z^\clf)^* &0 &0\\
 	0& M_z^\cle  &0\\
	0&0  & \bigoplus\limits_{k=1}^\infty J_k 
	\end{pmatrix}
	\begin{pmatrix}
	H_{\Phi_\clf} \\
	 T_{\Phi_\cle} \\
	 P_{\mathbb{E}} T_{\Phi_{\cle'}}
	\end{pmatrix}=\begin{pmatrix}
	(M_z^\clf)^*H_{\Phi_\clf}\\
	M_z^\cle T_{\Phi_{\cle}}\\
	\bigoplus\limits_{k=1}^\infty J_kP_{\mathbb{E}}T_{\Phi_{\cle'}}
	\end{pmatrix}=
	\begin{pmatrix}
	H_{\Phi_\clf}M_z^{\clf_0}\\
	T_{\Phi_{\cle}}M_z^{\clf_0}\\
	P_\mathbb{E}T_{\Phi_{\cle'}}M_z^{\clf_0}
	\end{pmatrix}=W_\Phi^* M_z^{\clf_0}.
\]
The second last equality holds because of the fact that  $M_z^{\cle'}$ on $H^2_{\cle'}(\D)$  is the minimal isometric dilation of $\bigoplus\limits_{k=1}^\infty J_k$ and hence 
\[
\left(\bigoplus\limits_{k=1}^\infty J_k\right)P_{\mathbb{E}}=P_{\mathbb{E}}M_z^{\cle'}.
\]
Hence $\clr(W_\Phi^*)$ is invariant under $T$. Also by Beurling-Lax-Halmos theorem, $\cln(T_\Theta^*) \oplus \{0\}\oplus\{0\}$ is  $T$-invariant. Thus $\clm$ is $T$-invariant. 
To sum up the above, we have the following result:

\begin{thm}\label{inv}
A closed subspace $\clm$ of $H^2_\clf(\D) \oplus H^2_\cle(\D) \oplus \left(\bigoplus\limits_{k=1}^\infty \mathbb{E}_k\right) $ is an invariant subspace for $T=(M_z^\clf)^* \oplus M_z^\cle \oplus \left(\bigoplus\limits_{k=1}^\infty J_k\right)$  if and only if 
\[
\clm=\bigvee\left\{\clr(W_\Phi^*), (\cln(T_\Theta^*)\oplus\{0\}\oplus \{0\})\right\},
\]
where $\mathbb{E}_k =\cle_k \otimes \C^k$, $W_\Phi=
\begin{bmatrix}
	H_{\Phi_\clf}^* & T_{\Phi_\cle}^* & T_{\Phi_{\cle'}}^*\big|_{\mathbb{E}}
\end{bmatrix}: H^2_\clf(\D) \oplus H^2_{\cle}(\D)\oplus \mathbb{E} \rightarrow H^2_{\clf_0}(\D)$ with $\cle'=\bigoplus\limits_{k=1}^\infty \cle_k$, $\Phi\in L^\infty_{\clb(\clf_0,\clf\oplus\cle\oplus\cle')}$ is either zero or $\Phi(e^{it})=
\begin{pmatrix}
\Phi_\clf({e^{it}})\\
\Phi_\cle(e^{it})\\
\Phi_{\cle'}(e^{it})
\end{pmatrix}$ for $\Phi_{\clf} \in L^\infty_{\clb(\clf_0,\clf)}$, $\Phi_\cle \in H^\infty_{\clb(\clf_0,\cle)}$, and $\Phi_{\cle'} \in H^\infty_{\clb(\clf_0,\cle')}$ satisfying  $\Phi(e^{it})^*\Phi(e^{it})=I_{\clf_0}$ a.e.\ on  $\T$, where $\clf_0 \subseteq \clf \oplus \cle \oplus \cle'$ and $\Theta \in H^\infty_{\clb(\clf_1,\clf)}$ is zero or inner and right extremal with $\clf_1 \subseteq \clf$.
\end{thm}

\begin{rem}
In particular, let $T=J_k$, a truncated shift of index $k$ on $\mathbb{E}_k$.  Then by above theorem, $W_\Phi=T_{\Phi}^*|_{\mathbb{E}_k}$, $\Phi \in H^\infty_{\clb(\clf_0,\cle_k)}$ such that $\Phi(e^{it})^*\Phi(e^{it})=I_{\clf_0}$  a.e.\ on $\T$. Therefore, a proper closed subspace $\clm$ is invariant under $J_k^*$ if and only if $\clm=\cln(T_{\Phi}^*) \cap \mathbb{E}_k$, where $\Phi$ is an inner function. Now since $J_k^*=M_z^*|_{\mathbb{E}_k}$, by Beurling-Lax-Halmos theorem, there exist  a Hilbert space $\clf_1 \subseteq \cle_k$ and $\Theta \in H^\infty_{\clb(\clf_1,\cle_k)}$ inner function such that 
\[
\cln(T_{\Phi}^*) \cap \mathbb{E}_k =H^2_{\cle_k}(\D) \ominus \Theta H^2_{\clf_1}(\D).
\]
\end{rem}

Using this fact, we  have a refined  and independent proof for characterizing the invariant subspaces of a truncated shift which is based on the Sz.-Nagy and Foia\c{s}' dilation.

\begin{thm}
Let $J_k$ be a truncated shift of index $k$ on $\mathbb{E}_k=\cle_k \otimes \C^k$ for $k \geq 1$. A  closed subspace $\clm$ of $\mathbb{E}_k$  is  invariant under $J_k$ if and only if there exist a Hilbert space $\clf_k $ isomorphic to $\cle_k $ and inner functions  $\Theta \in H^\infty_{\clb(\clf_k ,\cle_k )}$ and $ \Phi \in H^\infty_{\clb(\cle_k ,\clf_k )}$ such that $T_\Theta T_\Phi=M_{z^k}^{\cle_k} $ and $\clm =T_\Theta(\cln(T_\Phi^*))$.
\end{thm}

\begin{proof}
Suppose $\clm$ is an invariant subspace for $J_k$. Since $M_z^{\cle_k}$ on $H^2_{\cle_k}(\D)$ is the minimal isometric dilation of $J_k$ on $\mathbb{E}_k=[z^kH^2_{\cle_k} (\D)]^\perp$, so 
\[
J_k=P_{\mathbb{E}_k}M_z|_{\mathbb{E}_k} \text{ and } J_k^*=M_z^*|_{\mathbb{E}_k}.
\]
Now $J_k(\clm) \subseteq \clm$ infers  $J_k^*(\mathbb{E}_k \ominus \clm) \subseteq \mathbb{E}_k\ominus \clm$. Therefore, $M_z^*(\mathbb{E}_k\ominus \clm) \subseteq \mathbb{E}_k \ominus \clm$ and hence, by Beurling-Lax-Halmos theorem, there exist a Hilbert space $\clf_k  \subseteq \cle_k $ and an inner function $\Theta \in H^\infty_{\mathcal{B}(\clf_k ,\cle_k  )}$ such that $\mathbb{E}_k \ominus \clm =[\Theta H^2_{\clf_k }(\D)]^\perp$. Thus 
\[
\clm=\Theta H^2_{\clf_k }(\D) \cap \mathbb{E}_k=\Theta H^2_{\clf_k }(\D) \cap [z^k H^2_{\cle_k}(\D)]^\perp.
\]
Again,
$[\Theta H^2_{\clf_k }(\D)]^\perp \subseteq \mathbb{E}_k$ gives
\[
\mathbb{E}_k^\perp=M_{z^k}^{\cle_k}H^2_{\cle_k}(\D) \subseteq T_{\Theta} H^2_{\clf_k }(\D).
\]
By Douglas' range inclusion theorem \cite{DOUGLOUS}, there exists a contraction $X: H^2_{\cle_k}(\D) \rightarrow H^2_{\clf_k }(\D)$ such that $M_{z^k}^{\cle_k} =T_\Theta X$. 
Then
\[
T_\Theta XM_z^{\cle_k} = M_{z^k}^{\cle_k}  M_z^{\cle_k }=M_z^{\cle_k}  M_{z^k}^{\cle_k} =M_z^{\cle_k}  T_\Theta X=T_\Theta M_z^{\clf_k}   X.
\]
Therefore, $XM_z^{\cle_k }=M_z^{\clf_k }X$ and hence $X=T_\Phi$ for some $\Phi \in H^\infty_{\mathcal{B}(\cle_k  ,\clf_k )}$. Since $M_{z^k}^{\cle_k} $ and $T_\Theta$ are both isometries, $T_\Phi$ is an isometry from $H^2_{\cle_k}(\D)$ to $H^2_{\clf_k }(\D)$. Since $T_\Phi M_z^{\cle_k }=M_z^{\clf_k }T_\Phi$,  by \cite[Theorem 7.1]{AMS},
$\dim\,\cle_k  \leq \dim\,\clf_k $. But $\clf_k  \subseteq \cle_k  $ which implies $\dim\,\cle_k = \dim\,\clf_k $. In turn, $\cle_k  $ is isomorphic to $\clf_k $ and $T_\Theta T_\Phi=M_{z^k}^{\cle_k}$. Also, we have 
\begin{align*}
\clm =T_\Theta H^2_{\clf_k}  (\D)\ominus M_{z^k}^{\cle_k}  H^2_{\cle_k} (\D)
= T_\Theta H^2_{\clf_k}  (\D)\ominus T_\Theta T_\Phi H^2_{\cle_k} (\D).
\end{align*}
Then it is easy to see that 
\[
\clm=T_\Theta\left(\cln(T_\Phi^*)\right).
\]

Conversely, let there exist a Hilbert space $\clf_k $ isomorphic to $\cle_k $ and inner functions $\Theta \in H^\infty_{\clb(\clf_k ,\cle_k )}$ and $ \Phi \in H^\infty_{\clb(\cle_k ,\clf_k )}$ such that $T_\Theta T_\Phi=M_{z^k}^{\cle_k} $ and  $\clm=T_\Theta\left(\cln(T_\Phi^*)\right)$. Then for $x \in \cln(T_\Phi^*)$,
\begin{align*}
		J_k\left(T_\Theta x\right) &=P_{\mathbb{E}_k} M_z^{\cle_k }T_\Theta x\\
					&= \left(I-M_{z^k}^{\cle_k}  (M_{z^k}^{\cle_k} )^*\right)T_\Theta M_z^{\clf_k } x\\
					&=  \left(I- T_\Theta T_\Phi (M_{z^k}^{\cle_k} )^*\right)T_\Theta M_z^{\clf_k } x\\		
					&=T_\Theta\left( I-T_\Phi T_\Phi^*\right)M_z^{\clf_k } x\\
					&\subseteq T_\Theta\left(\cln(T_\Phi^*)\right).
\end{align*}
Hence $\clm$ is invariant under $J_k$. This completes the proof.
\end{proof}

\begin{rem}
In Theorem \ref{inv}, take  $\clf_0=\clf\oplus \cle\oplus\cle'$ and an  inner and right extremal function $\Theta \in H^\infty_{\clb(\clf_1,\clf)}$ with $\clf_1 \subseteq \clf$. Define $\Phi: H^2_\clf(\D) \oplus H^2_\cle(\D) \oplus \left(\bigoplus\limits_{k=1}^\infty H^2_{\cle_k}(\D)\right) \rightarrow H^2_{\clf_0}(\D)$ by
\[
\Phi=\begin{pmatrix}
z\Phi_\clf& & & &\\
& \Phi_\cle & & &\\
& & \Phi_{1} & & \\
& & & \Phi_{2} & \\
& & & &\ddots
\end{pmatrix},
\]
where $\Phi_\clf$, $\Phi_\cle$, and $\Phi_k\,\, (k \geq 1)$ are inner functions in $H^\infty_{\clb(\clf)}$, $H^\infty_{\clb(\cle)}$, and $H^\infty_{\clb(\cle_k)}$ respectively. Then,
\[
W_\Phi=
\begin{pmatrix}
H_{z\Phi_\clf}^* &&&&\\
& T_{\Phi_\cle}^*&&&\\
&&T_{\Phi_{1}}^*\big|_{\mathbb{E}_1}&&\\
&&&T_{\Phi_{2}}^*\big|_{\mathbb{E}_2}&\\
&&&&\ddots
\end{pmatrix}.
\]
Then $\cln(W_\Phi)=H^2_\clf(\D) \oplus \cln(T_{\Phi_{\cle}}^*) \oplus \bigoplus\limits_{k=1}^\infty \cln(T_{\Phi_{k}}^*\big|_{\mathbb{E}_k})$. As stated earlier, for each $k \geq 1$, there exist Hilbert spaces $\clf_k$ and inner functions $\Theta_k \in H^\infty_{\clb(\clf_k,\cle_k)}$ such that
\[
\cln(T_{\Phi_{k}}^*\big|_{\mathbb{E}_k}) =H^2_{\cle_k}(\D) \ominus \Theta_k H^2_{\clf_k}(\D).
\]
Hence,
\[
\clr(W_\Phi^*)=\{0\} \oplus \clr(T_{\Phi_{\cle}}) \oplus \bigoplus\limits_{k=1}^\infty \left(\clr(T_{\Theta_k}) \cap \mathbb{E}_k\right).
\]
So, $\clm=\cln(T_\Theta^*) \oplus \clr(T_{\Phi_{\cle}}) \oplus \bigoplus\limits_{k=1}^\infty \left(\clr(T_{\Theta_k}) \cap \mathbb{E}_k\right)$. In this case, $\clm$ is a splitting subspace of $T$.
\end{rem}
\vspace{.3 cm}

\section{Hyperinvariant subspaces of a c.n.u. power partial isometry}

In this section, we characterize all the hyperinvariant subspaces of a c.n.u.\ power partial isometry. We will use the following lemmas repeatedly in the main proof. The first lemma was given by Douglas and Pearcy in \cite{douglastopology}.

\begin{lem}\label{Lemma1}
Let $\clh_1$ and $\clh_2$ be two Hilbert spaces and $T \otimes I_{\clh_2} \in \clb(\clh_1 \otimes \clh_2)$. A closed subspace $\clm$ is hyperinvariant for $T \otimes I_{\clh_2}$ if and only if it is of the form $\clm=\cls \otimes \clh_2$, where $\cls$ is hyperinvariant for $T$.
\end{lem}

The following lemma was given by Wu in  \cite{wuhyperinvariant} for hyperinvariant subspaces of direct sum of operators.

\begin{lem}\label{thm-hypinv-directsum}
Let $\mathcal{H}_1$ and $\mathcal{H}_2$ be two Hilbert spaces and $T_1  \in \mathcal{B}(\clh_1)$ and $T_2 \in \mathcal{B}(\clh_2)$. Then $\clm \subseteq \mathcal{H}_1 \oplus \mathcal{H}_2$ is hyperinvariant for $T_1 \oplus T_2$ if and only if $\clm=\clm_1 \oplus \clm_2$, where $\clm_1 \subseteq \clh_1$ and $\clm_2 \subseteq \clh_2$ are hyperinvariant for $T_1$ and $T_2$, respectively and $S_1\clm_1 \subseteq \clm_2$ and $S_2\clm_2 \subseteq \clm_1$ for any operators $S_1 \in \mathcal{B}(\clh_1, \clh_2)$ and $S_2 \in \mathcal{B}(\clh_2, \clh_1)$ satisfying $S_1T_1 = T_2S_1$ and $S_2T_2 = T_1S_2$.
\end{lem}

The following result is of independent interest and useful in the sequel. For the sake of completeness, we write the detailed proof.

\begin{lem}\label{lem-innerfunction}
Let $u$ be a scalar-valued inner function. Then $u$ is a polynomial if and only if $u$ is a monomial. 
\end{lem}

\begin{proof}
The sufficiency of this condition is trivial. We use mathematical induction to prove the necessity part. 

Let $u$ be a polynomial of degree $1$. 
Then $u(z)=u_0+u_1z$, $u_0,u_1 \in \C$. Since $u$ is inner, $\overline{u(e^{it} )}u(e^{it})=1$ a.e. on $\T$. Consider
\[
\overline{u(e^{it} )}u(e^{it})=\vert u_0\vert^2 +\vert u_1 \vert^2 
+ \overline{u_0}u_1e^{it} + \overline{u_1}u_0e^{-it}=1
\]
if and only if $\overline{u_0}u_1=0$. Hence either $u_0=0$ or $u_1= 0$ which implies either $u(z)=u_0,\,\, \vert u_0\vert=1$ or $u(z)=u_1z,\,\,\vert u_1\vert=1$. Therefore, $u$ is a monomial.

Now assume that if an inner function $u$  is a polynomial of degree $n-1$, then it is a monomial. Suppose that $u(z)=\sum\limits_{m=0}^n u_mz^m$ is a polynomial of degree $n$ and also inner. Then
\[
\overline{u(e^{it})}u(e^{it})=\sum\limits_{m=0}^n \sum\limits_{k=0}^n \overline{u_m}u_k e^{i(k-m)t}.
\]
Since $u$ is inner, all the coefficients except the constant term are zero. Now on comparing the coefficient of $e^{int}$, we get $\overline{u_0}u_n=0$. Hence either $u_0=0$ or $u_n=0$. If $u_n=0$, then $u$ is a polynomial of degree  $n-1$ and therefore, by induction hypothesis, it is a monomial of degree at most $n-1$. If $u_n\neq 0$, then $u_0=0$. Thus $u = zv$, where $v$ is an inner function of degree
at most $n-1$. Again by induction hypothesis, $v$ is a monomial of degree at most $n-1$ and hence $u$ is a monomial of degree at most $n$.
\end{proof}

The following result may be well-known to the experts, but for the sake of being complete, we give an independent proof based on the Sz.-Nagy and Foia\c{s} dilation.

\begin{lem}\label{thm-hypinvtruncated}
The only hyperinvariant subspaces of a truncated shift $J_k$ on $\cle \otimes \C^k$ of index $k$ are $\cln(J_k^n),\, n=0,1,\ldots,k$.
\end{lem}

\begin{proof}
It is trivial to check that, for each $n \in \{0,1,\ldots, k\}$, $\cln(J_k^n)$ is hyperinvariant under $J_k$. 

For the converse,  note that $J_k$ can be identified as $I_{\cle} \otimes S$, 
where $S$ is the right shift operator on $\C^k$. So, by Lemma \ref{Lemma1}, a hyperinvariant subspace of $J_k$ is of the form $\E \otimes \clm$, where $\clm \subseteq \C^k$  is a hyperinvariant subspace for $S$.
Define $\Pi: \C^k \rightarrow [z^kH^2(\D)]^\perp$ as 
\[
\Pi(a_0,a_1,\ldots, a_{k-1})=\sum\limits _{m=0}^{k-1}a_mz^m.
\]
Clearly, $\Pi$ is a unitary map and hence $\C^k$ can be identified as $\mathcal{H}=[z^kH^2(\D)]^\perp = H^2(\D) \ominus z^kH^2(\D)$. Since $M_z$ on $H^2(\D)$ is the minimal isometric dilation of $S$, we write $S^*=M_z^*|_{\mathcal{H}}$, where $\clh$ is $M_z^*$ invariant subspace. Suppose $\clm$ is hyperinvariant for $S$. Then $S\clm \subseteq \clm$, that is, $S^*(\mathcal{H}\ominus \clm)\subseteq \mathcal{H} \ominus \clm$ which implies $M_z^*(\mathcal{H}\ominus \clm)\subseteq \mathcal{H} \ominus \clm$. Now by Beurling's theorem, there exists an inner function $u \in H^\infty$ such that 
\[
\mathcal{H}\ominus \clm=[u\h]^\perp.
\]
Therefore, 
\[
\clm=u\h \cap \mathcal{H}=u\h \ominus z^k\h.
\] 
Now $[u\h]^\perp \subseteq \mathcal{H}$ yields $z^k\h\subseteq u\h$. Thus there exists an inner function $v \in H^\infty$ such that $z^k=uv$. Therefore, 
\[
M_{z^k}=T_uT_{v} \Rightarrow T_u^*M_{z^k}=T_{v}. 
\]
Hence $\bar{u}z^k=v \in H^\infty$ which implies $u=\sum\limits_{m=0}^{k}u_mz^m$, where $u_m \in \C$ for all $0 \leq m\leq k$. Since $u$ is inner, using Lemma \ref{lem-innerfunction}, we get $u=u_nz^n$ for some $0\leq n \leq k$, where $\vert u_n\vert=1$. Thus
\[
\clm= z^n\h\cap \clh=\clr ( S^n)=\cln( S^{k-n}).
\]
Hence, any hyperinvariant subspace of $J_k$ is of the form $\E \otimes \cln (S^{n})=\cln (J_k^n)$ for some $n=0,1,\ldots,k$. 
\end{proof}

The next goal is to characterize the hyperinvariant subspaces of the direct sum of infinite truncated shifts and the above theorem will be useful. However, the technique is motivated by the above Lemma \ref{thm-hypinv-directsum}. Before proceeding to the proof of the next result, we need the following observation.

Recall that $\mathbb{E}_k = \cle_k \otimes \C^k$ for $k \geq 1$ and let $S \in \clb\left(\bigoplus\limits_{k=1}^\infty \mathbb{E}_k\right)$ such that
\begin{equation}\label{eq2}
S\left(\bigoplus\limits_{k=1}^\infty J_k\right)=\left(\bigoplus\limits_{k=1}^\infty J_k\right)S.
\end{equation}
Let $[S_{ij}]$ be the block matrix representation of $S$, where $S_{ij} \in \clb(\mathbb{E}_j,\mathbb{E}_i)$. Then (\ref{eq2}) holds if and only if
\begin{equation}\label{eq3}
J_i S_{ij}=S_{ij}J_j \quad (\forall~~ i, j \geq 1).
\end{equation}
Now write $S_{ij}$ as a block matrix of order $i \times j$; then it is not difficult to see by (\ref{eq3}) that whenever $i \leq j$,
$S_{ij}=\begin{bmatrix}
X_{i \times i}&O_{i \times (j-i)}
\end{bmatrix}$. 
Similarly, for $i>j$, $S_{ij}=\begin{pmatrix}
O_{(i-j)\times j}\\
Y_{j \times j}
\end{pmatrix}$, 
where $X$ and $Y$ are lower triangular block Toeplitz matrices whose entries are operators from $\cle_j$ to $\cle_i$. Also note that, for any $j\geq 1$, 
\[
\cln(J_j^{n_j})=\{({0,\ldots,0}, {a_1},a_2,\ldots,a_{n_j}): a_n \in \cle_j, 1 \leq n \leq n_j\}.
\]
Hence, it is clear that the first nonzero entry of any element of $\cln(J_j^{n_j})$ is at $(j-n_j+1)^{th}$ position or in the right of it. We frequently use this fact in the proof of the following result.

\begin{lem}\label{lemma-hyperinv-infdirect}
For each $k\geq 1$, let $J_k$ be the truncated shift of index $k$ on $\mathbb{E}_k=\cle_k \otimes \C^k$. Then the only hyperinvariant subspaces of $\bigoplus\limits_{k=1}^\infty J_k \in \clb\left(\bigoplus\limits_{k=1}^\infty \mathbb{E}_k \right)$ are $\bigoplus\limits_{k=1}^\infty \cln(J_k^{n_k })$, where $n_k  \in \{0,1,\ldots,k\}$ and $n_k  \leq n_{k+1} \leq n_k +1$.
\end{lem}

\begin{proof}
Let $R=\bigoplus\limits_{k=1}^\infty J_k$ on $\mathbb{E}=\bigoplus\limits_{k=1}^\infty \mathbb{E}_k $ and $\clm$ be a hyperinvariant subspace of $R$. For each $j \geq 1$, let $I_j$ be the identity operator on $\mathbb{E}_j $.

Define an operator $I_j'$ on 
$\mathbb{E}$ as
\[
I_j'=\bigoplus\limits_{k=1}^\infty \Delta_k \textrm{, where } \Delta_k= 
\begin{cases}
I_j,& k=j\\
0, & k \neq j
\end{cases}.
\]
Then $I_j' R=RI_j'$. Therefore, $\clm$ is invariant under $I_j'$, that is, $P_j (\clm) \subseteq \clm$ where $P_j$ is the orthogonal projection of $\mathbb{E}$ onto $\mathbb{E}_j $. Thus $\clm=\bigoplus\limits_{k=1}^\infty P_k(\clm)$. Note that, for each $j \geq 1$,
\[
J_jP_j(\clm)=RI_j'(\clm)=I_j'R(\clm) \subseteq I_j'(\clm)=P_j(\clm). 
\]
Clearly, $ P_j(\clm) \subseteq \mathbb{E}_j $ is hyperinvariant for $J_j$. Therefore, any hyperinvariant subspace  $\clm$ of $R$ must be of the form $\bigoplus\limits_{k=1}^\infty \clm_k$, where $\clm_k$ is hyperinvariant for $J_k$. Now using Lemma \ref{thm-hypinvtruncated}, $\clm=\bigoplus\limits_{k=1}^\infty \cln (J_k^{n_k })$ where $n_k  \in\{0,1, \ldots,k\}$. Being $\clm$ hyperinvariant, for any operator $S$ commuting with $R$, we have $S(\clm) \subseteq \clm$. But from the remarks before this lemma, it suffices to consider such $S_{ij}$'s for which (\ref{eq3}) holds.

\textbf{Case-1:} For $i> j$, define an operator $S_{ij}: \mathbb{E}_j \rightarrow \mathbb{E}_i  $ in the following way:
\[
S_{ij}=
\begin{pmatrix}
& & & & \\
&  & & & \\
& & {\Huge O}_{(i-j) \times j} & &\\
& & & &\\
\theta_{1}& 0& \cdots &\cdots&0\\
\theta_2 &\theta_1& 0& \cdots &0\\
\theta_3&\theta_2&\theta_1 & \cdots &0\\
\vdots&\vdots&\cdots&\ddots&\vdots\\
\theta_j& \theta_{j-1}& \cdots&\cdots&\theta_{1}
\end{pmatrix}_{i \times j},
\]
where ${\theta_m}\in \mathcal{B}(\cle_j,\cle_i)$ for $1 \leq m\leq j $. Then $J_iS_{ij}=S_{ij}J_j$. Now define  block operator $S_{ij}'$ on  $\mathbb{E}$ in such a way that its $ij^{th}$ entry is $S_{ij}$ and others are zero. Then $RS_{ij}'=S_{ij}'R$. Therefore,
\[
S_{ij}'(\clm)\subseteq \clm \Leftrightarrow S_{ij}( \cln (J_j^{n_j})) \subseteq \cln (J_i^{n_i}). 
\]
If $n_j=0$, then it is trivially true. If $n_j \neq 0$, then  let $(0,\ldots,0, \bm{a_1},a_2,\ldots,a_{n_j}) \in \cln (J_j^{n_j})$, where $a_n \in \cle_j$ for all $ 1 \leq n\leq n_j$, $a_1 \neq 0$ and the boldface shows the $(j-n_j+1)^{th}$ position. Then  
\begin{equation}
S_{ij}(0,\ldots,0, \bm{a_1},a_2,\ldots,a_{n_j}) \in \cln (J_i^{n_i}) \Leftrightarrow \left(0, \ldots,0, \bm{\theta_1( a_1)}, \ldots, \sum\limits_{m=1}^{n_j}\theta_{n_j-m+1}(a_m)\right) \in  \cln (J_i^{n_i}) \label{eq4}
\end{equation}
where boldface in the latter inclusion shows the $(i-n_j+1)^{th}$ position. But this is possible for all $\theta_1 \in \clb(\cle_j, \cle_i)$ if and only if $n_j \leq n_i$. Here we used  the fact that $\clx \ni x=0$  if and only if $Ax=0$ for all $ A \in \clb(\clx,\cly)$, where $\clx$ and $\cly$ are Banach spaces.

Note that if $n_i=0$, then (\ref{eq4})  will imply
\[
(0, \ldots,0, \bm{\theta_1 (a_1)}, \ldots, \sum\limits_{m=1}^{n_j}\theta_{n_j-m+1}(a_m))=0,
\]
for all $\theta_1 \in \clb(\cle_j, \cle_i)$  which is not  possible as $a_1 \neq 0$. More precisely, if $n_i=0$, then $n_j=0$. So, 
$n_j \leq n_i$ still holds. 

\textbf{Case-2:} Let $i \leq j$. 
Define $S_{ij}: \mathbb{E}_j\rightarrow \mathbb{E}_i$ in the following way:
\[
S_{ij}=\begin{pmatrix}
\theta_1 &0&0&\cdots&\cdots&0\\
\theta_2 &\theta_1&0&\cdots&\cdots&0\\
\vdots&\vdots&\ddots&\cdots&\cdots&\vdots\\
\theta_i&\theta_{i-1}&\cdots&\theta_1&\cdots&0
\end{pmatrix}_{i\times j},
\]
where $\theta_m\in \mathcal{B}(\cle_j,\cle_i)$ for all $1 \leq m \leq i$. Then $J_iS_{ij}=S_{ij}J_j$. Now define $S_{ij}'$ on  $\mathbb{E}$ as above. Then again, $S_{ij}(\cln (J_j^{n_j}))\subseteq \cln (J_i^{n_i})$. If $n_j=0$, then by Case-1, $n_i=0$. For $ n_j \neq 0$, let $(0,\ldots,0, \bm{a_1},a_2,\ldots,a_{n_j}) \in \cln (J_j^{n_j})$, where $a_n \in \cle_j$ for all $1\leq n\leq n_j$, $a_1\neq 0$ and  boldface shows the $(j-n_j+1)^{th}$ position. If $j-n_j+1 \leq i$, let $l=n_j-j+i$, then 
\[
S_{ij}(0,\ldots,0, \bm{a_1},a_2,\ldots,a_{n_j}) \in \cln (J_i^{n_i})\Leftrightarrow (0, \ldots,0, \bm{\theta_1 (a_1)}, \ldots, \sum\limits_{m=1}^{l}\theta_{l-m+1}(a_m)) \in  \cln (J_i^{n_i}).
\]
where the boldface in the latter inclusion is at the $(j-n_j+1)^{th}$ position.
Again this is possible for all $\theta_1 \in \clb(\cle_j,\cle_i)$ if and only if $j-n_j+1 \geq i-n_i+1$.

On the other hand, $j-n_j+1>i$ if and only if $S_{ij}(0,\ldots,0, \bm{a_1},a_2,\ldots,a_{n_j})=0$ for all $a_n \in \cle_j \,\,(1\leq n\leq n_j)$ and for all $S_{ij}$ defined above. In that case also, $j-n_j+1 \geq i-n_i+1.$  In particular, let $j=i+1$, we get $n_{i+1} \leq n_i+1$. Also from the former case, $n_i \leq n_{i+1}$. 

Hence any hyperinvariant subspace $\clm$  of $R$ is of the form 
$\clm=\bigoplus\limits_{k=1}^\infty \cln (J_k^{n_k})$, where $n_k \in\{0,1,\ldots,k\}$ and $n_k  \leq n_{k+1} \leq n_k +1$. 

The converse holds from the above proof and the observation before this lemma.
\end{proof}

Let $T$ be a c.n.u.\ power partial isometry. We will now characterize hyperinvariant subspaces of a pure power partial isometry, i.e., there is no co-isometric part. We will then use it to characterize hyperinvariant subspaces of a c.n.u.\ power partial isometry. Before that, first recall the following result.

\begin{lem}[cf. \cite{douglashyperinvariant}]\label{hyper-shift}
A non-trivial closed subspace of $H^2_\cle(\D)$ is hyperinvariant for $M_z$ if and only if it is of the form $uH^2_\cle(\D)$ for some scalar-valued inner function $u$.
\end{lem}

\begin{thm}
Let $\clm$ be a closed subspace of $H^2_\cle(\D) \oplus \left(\bigoplus\limits_{k=1}^\infty \mathbb{E}_k\right)$. Then $\clm$ is hyperinvariant for pure power partial isometry $M_z^\cle \oplus \left(\bigoplus\limits_{k=1}^\infty J_k\right)$ if and only if either
$\clm= \{0 \} \oplus \bigoplus\limits_{k=1}^\infty \cln (J_k^{n_k})$ or  $\clm=u\he \oplus \bigoplus\limits_{k=1}^\infty \cln (J_k^{n_k})$ for some $n_k  \in \{0,1,2,\ldots,k\}$ such that $n_k \leq n_{k+1} \leq n_k +1$ and a scalar-valued inner function $u$ such that $u \in z^{k-n_k}H^2(\D)$ for all $k \geq 1$. 
\end{thm}

\begin{proof}
Let $\cle'=\bigoplus\limits_{k=1}^\infty \cle_k$,  $\mathbb{E}=\bigoplus\limits_{k=1}^\infty \mathbb{E}_k$ and $\mathbb{E}_k=[z^kH^2_{\cle_k}(\D)]^\perp$ (see Section 4). Also $J_k$ is defined as in (\ref{def}). Set 
\begin{align*}
R & =\bigoplus\limits_{k=1}^\infty J_k
\mbox{~~~and~~} \\
V & =M_z^\cle \oplus R \mbox{~~on~~} \he \oplus \mathbb{E}.
\end{align*}
Suppose $\clm$ is hyperinvariant for $V$.  Then, from Lemma \ref{thm-hypinv-directsum}, $\clm=\clm_1\oplus \clm_2$ such that $\clm_1 \subseteq \he$ and $\clm_2 \subseteq \mathbb{E}$ are hyperinvariant for $M_z^\cle$ and $R$,  respectively and for any operators $S_1 \in \mathcal{B}(\he,\mathbb{E})$ and $S_2 \in \mathcal{B}\left(\mathbb{E},\he\right)$ satisfying $ RS_1=S_1M_z^\cle$ and $M_z^\cle S_2=S_2R$, we have $S_1\clm_1 \subseteq \clm_2$ and $S_2\clm_2 \subseteq \clm_1$. By  Lemma \ref{lemma-hyperinv-infdirect} and Remark \ref{hyper-shift}, $\clm$ has two possibilities:
\[
\clm= \{0\} \oplus \left(\bigoplus\limits_{k=1}^\infty \cln (J_k^{n_k })\right) \textrm{ or } \clm=u\he \oplus \left( \bigoplus\limits_{k=1}^\infty \cln (J_k^{n_k })\right),
\]
where $u\in H^\infty$ is inner and $n_k  \in \{0,1,\ldots,k\}$ such that $n_k  \leq n_{k+1}\leq n_k +1$.  

First let $S_1 \in \mathcal{B}(\he,\mathbb{E})$ such that $RS_1=S_1M_z^\cle$. Since the minimal isometric dilation of $R$ is $M_z^{\cle'}$ on $H^2_{\cle'}(\D)$, it follows by commutant lifting theorem that there exists a contraction $Y : \he \rightarrow H^2_{\cle'}(\D)$ such that
\begin{enumerate}
\item $M_z^{\cle'} Y=YM_z^{\cle}$ \label{(1)},
\item $S_1 =P_{\mathbb{E}}Y,$ and \label{(2)}
\item $\Vert S_1\Vert=\Vert Y \Vert$,
\end{enumerate}
where $P_\mathbb{E}$ is the orthogonal projection of $H^2_{\cle'}(\D)$ onto $\mathbb{E}$.

Now (\ref{(1)}) implies $Y=T_\Phi$ for some $\Phi \in H^\infty_{\mathcal {B}(\cle, \cle')}$ and hence (\ref{(2)}) infers $S_1=P_\mathbb{E}T_\Phi$. Conversely, let $S_1=P_{\mathbb{E}}T_\Phi$ for some $\Phi \in H^\infty_{\mathcal {B}(\cle, \cle')}$. Then
\[
RS_1=RP_{\mathbb{E}}T_\Phi=P_{\mathbb{E}}M_z^{\cle'} T_\Phi=P_{\mathbb{E}}T_\Phi M_z^\cle=S_1M_z^\cle.
\]
Hence, $RS_1=S_1M_z^\cle$ if and only if $S_1= P_\mathbb{E} T_\Phi$, where $\Phi \in H^\infty_{\clb(\cle,\cle')}$. 

For $\clm=\{0\} \oplus \left(\bigoplus\limits_{k=1}^\infty \cln(J_k^{n_k})\right), S_1(\{0\}) \subseteq \bigoplus\limits_{k=1}^\infty \cln(J_k^{n_k})$ is trivially true.

For $\clm=uH^2(\D) \oplus \left(\bigoplus\limits_{k=1}^\infty \cln(J_k^{n_k})\right)$ being  hyperinvariant, for each $ \Phi \in H^\infty_{\mathcal{B}(\cle, \cle')},$
\begin{align*}
P_{\mathbb{E}}T_\Phi(u\he)& \subseteq \bigoplus\limits_{k=1}^\infty \cln (J_k^{n_k }) \\
\Leftrightarrow \bigoplus\limits_{k=1}^\infty P_{\mathbb{E}_k }T_\Phi(u \he) &\subseteq \bigoplus\limits_{k=1}^\infty \cln (J_k^{n_k })\\
\Leftrightarrow P_{\mathbb{E}_k }T_\Phi(u \he) &\subseteq \cln (J_k^{n_k }), \textrm{ for all } k \geq 1.
\end{align*}
Now let $u=\sum\limits_{m=0}^\infty u_mz^m$ and, for fixed $k\geq 1$, choose $\Phi \in \mathcal{B}(\E,\E_k)$ and $a \in \E$ such that $\Phi(a) \neq 0$. Then 
\[
P_{\mathbb{E}_k}T_\Phi(u a)=P_{\mathbb{E}_k}\left(\sum\limits_{m=0}^\infty u_m \Phi(a)z^m\right)=\sum\limits_{m=0}^{k-1}u_m\Phi(a)z^m.
\]

If $n_k  <k$ for some $k$, then  $P_{\mathbb{E}_k}T_\Phi(u a) \in \cln (J_k^{n_k }) \Leftrightarrow u_m=0\,\,  \forall \,\,m<k-n_k$.  

For $n_k=k$, the condition holds trivially. Hence, $u \in z^{k-n_k}H^2(\D)$ for all $k \geq 1$.

Conversely, suppose $\clm=uH^2(\D) \oplus \left(\bigoplus\limits_{k=1}^\infty \cln(J_k^{n_k})\right)$, where $u \in H^\infty$ is an inner function such that $u \in z^{k-n_k}H^2(\D)$ and $n_k \leq n_{k+1}\leq n_k+1$ for all $k \geq 1$. Recall that  for $k \geq 1$, we write 
\[
\mathbb{E}_k=[z^kH^2_{\E_k}(\D)]^\perp \subseteq H^2_{\E_k}(\D). 
\]
Hence, $P_{\mathbb{E}_k}=P_{\mathbb{E}_k}P_{H^2_{\E_k}(\D)}$ where $P_{H^2_{\cle_k}(\D)}$ is the orthogonal projection of $H^2_{\cle'}(\D)$ onto $H^2_{\E_k}(\D)$.
Then, for any $\Phi \in H^\infty_{\clb(\E,\E')}$, 
\begin{align*}
P_{\mathbb{E}_k} T_\Phi(u\he) & \subseteq P_{\mathbb{E}_k} T_\Phi(z^{k-n_k}\he)\\
&=P_{\mathbb{E}_k} P_{H^2_{\E_k}(\D)}T_\Phi(z^{k-n_k}\he)\\
&= (I-(M_z^{\E_k})^k(M_z^{\E_k})^{*k})P_{H^2_{\E_k}(\D)}T_\Phi (M_z^\E)^{k-n_k} \he\\
 &= (I-(M_z^{\E_k})^k(M_z^{\E_k})^{*k}) (M_z^{\E_k})^{k-n_k}P_{H^2_{\E_k}(\D)} T_\Phi \he\\
&= M_z^{k-n_k} (I-M_z^{n_k}M_z^{*n_k})P_{H^2_{\E_k}(\D)} T_\Phi \he\\
&\subseteq \cln(J_k^{n_k}).
\end{align*}
Hence 
\[
P_\mathbb{E}  T_\Phi(u\he) =\bigoplus\limits_{k=1}^\infty  P_{\mathbb{E}_k} T_\Phi(uH^2_\E(\D)) \subseteq \bigoplus\limits_{k=1}^\infty \cln(J_k^{n_k}).
\]

Now let $S_2 \in \mathcal{B}\left(\mathbb{E},\he\right)$ such that $M_z^\cle S_2=S_2R$. Then, for $n \geq 1$,
\[
(M_z^\cle)^nS_2=S_2R^n. 
\]
Since $R^n \rightarrow 0$ as $n \rightarrow \infty$ in strong operator topology, thus $S_2=0$.

Therefore, only  hyperinvariant subspaces of $V$ are those  of the form
\[
\clm= \{0 \} \oplus \left(\bigoplus\limits_{k=1}^\infty \cln (J_k^{n_k })\right)
\]
or
\[
\clm=u\he \oplus \left(\bigoplus\limits_{k=1}^\infty \cln (J_k^{n_k})\right), 
\]
where $n_k \in \{0,1,\ldots,k\}$ such that $n_k  \leq n_{k+1} \leq n_k +1$ and $u\in H^\infty$ inner such that $u \in z^{k-n_k}H^2(\D)$ for all $k \geq 1$.

This finishes the proof.
\end{proof}

\begin{ex}
Let $V=M_z \oplus \left(\bigoplus\limits_{k=1}^\infty J_k\right)$ on $H^2(\D) \oplus \left(\bigoplus\limits_{k=1}^\infty \C^k\right)$. Consider
\[
\clm= zH^2(\D) \oplus \left(\bigoplus\limits_{k=1}^\infty \cln(J_k)\right).
\]
Define 
$S: H^2(\D)  \rightarrow  \bigoplus\limits_{k=1}^\infty H^2(\D)$ as
\[ 
Sf= 0\oplus 0 \oplus 0 \oplus  z f \oplus 0\oplus \cdots \quad (f \in H^2(\D) ).
\]
Set $\clk=\bigoplus\limits_{k=1}^\infty \C^k$. Now let 
\[
S'= \begin{pmatrix}
0 &0\\
 P_{\clk} S &0
\end{pmatrix}: 
\begin{matrix}
 \begin{array}{ccc}
H^2(\D) & &H^2(\D)\\
\oplus & \rightarrow& \oplus\\
\clk & &\clk
\end{array}
\end{matrix},
\]
where $P_{\clk}$ is the orthogonal projection from $\bigoplus\limits_{k=1}^\infty H^2(\D)$ onto $ \clk$. Then it is easy to see that $S'V=VS'$ but $S' \clm \nsubseteq \clm$. Indeed, for $a \in \C\setminus \{0\}$,
\[
S'(az \oplus 0)=0 \oplus P_\clk (0 \oplus 0 \oplus 0 \oplus  az^2 \oplus \cdots) =0\oplus  (0 \oplus 0 \oplus 0 \oplus az^2 \oplus 0 \oplus \cdots).
\]
But $az^2 \notin \cln(J_4)$. Hence $\clm$ is not hyperinvariant. 

In the same space, if we take $\clm= zH^2(\D) \oplus\left( \bigoplus\limits_{k=1}^\infty \cln(J_k^{k-1})\right)$, then $\clm$ is hyperinvariant for $V$.
\end{ex}

We are now ready to find the explicit structure of hyperinvariant subspaces of a c.n.u.\ power partial isometry. For simplicity, we will do it in two parts. First, we will find the hyperinvariant subspaces of direct sum of forward and backward shift operator and then that of direct sum of backward shift and truncated shift.

\begin{thm}
A closed subspace of $H^2_\cle(\D) \oplus H^2_\clf(\D)$ is hyperinvariant for $M_z^\cle \oplus (M_z^\clf)^*$ if and only if it is either of the form $\{0 \} \oplus [vH^2_{\clf}(\D)]^\perp$, $ \{0\}\oplus H^2_\clf(\D) $ or $uH^2_\cle(\D) \oplus H^2_\clf(\D)$, where $u$ and $v$ are scalar-valued inner functions.
\end{thm}

\begin{proof}
Let $\clm$ be hyperinvariant for $X=M_z^\cle \oplus (M_z^\clf)^*$. Then from Lemma \ref{thm-hypinv-directsum}, there are four possible types, namely, 
\begin{align*}
\clm &=  \{0 \} \oplus [vH^2_{\clf}(\D)]^\perp, \\
\clm &=  uH^2_\cle(\D) \oplus [vH^2_{\clf}(\D)]^\perp, \\
\clm &=  \{0 \} \oplus H^2_\clf(\D), \\ \mbox{~~or~~}
\clm &=   uH^2_\cle(\D) \oplus H^2_\clf(\D),
\end{align*}
where $u,v \in H^\infty$ are inner functions. 

Let $S_1: H^2_{\E}(\D) \rightarrow H^2_{\clf}(\D)$ such that 
$(M_z^\clf)^*S_1=S_1M_z^\cle$. This holds if and only if
\[
S_1=H_{\Phi} \mbox{~~for some~~} \Phi \in L^\infty_{\mathcal {B}(\E,\clf)}.
\] 

For $\clm= \{0 \} \oplus [vH^2_{\clf}(\D)]^\perp$, $\{0 \} \oplus H^2_\clf(\D)$ or $uH^2_\cle(\D) \oplus H^2_\clf(\D)$, the condition given in Lemma \ref{thm-hypinv-directsum} holds trivially.
For 
\[
\clm=uH^2_\E(\D) \oplus [vH^2_{\clf}(\D)]^\perp
\]
being hyperinvariant, 
\[
H_\Phi(uH^2_{\E}(\D)) \subseteq [vH^2_{\clf}(\D)]^\perp \,\, ~~ \forall \,\, \Phi \in L^\infty_{\mathcal {B}(\E,\clf)}.
\]
In particular, let $\Phi(e^{it})=v(e^{-it})\Phi'(e^{it})\overline{u(e^{it})}$, where $\Phi' \in L^\infty_{\mathcal{B}(\E,\clf)}$ such that all the non-negative Fourier coefficients of $\Phi'$ are zero.
Then 
\[ 
H_\Phi(uH^2_{\E}(\D)) \subseteq [vH^2_{\clf}(\D)]^\perp
\]
then, for any $a \in \mathcal{E}$, we have
\begin{align*}
&\langle P_+^{\clf}J(\Phi ua),vg \rangle =0 ~~\,\, \forall \,\,g \in H^2_\clf(\D)\\
\Leftrightarrow & \langle J(\Phi ua),vg \rangle =0  \,\, \forall\,\, g \in H^2_\clf(\D)\\
\Leftrightarrow&\frac{1}{2\pi}\int_{0}^{2\pi}\left\langle v(e^{it})\Phi'(e^{-it})a, v(e^{it})g(e^{it}) \right\rangle_\clf=0 \,\, \forall \,\, g \in H^2_\clf(\D)  \\
\Leftrightarrow &  \frac{1}{2\pi}\int_{0}^{2\pi}\left\langle \Phi'(e^{-it})a, g(e^{it}) \right\rangle=0 \,\, \forall \,\, g \in H^2_\clf(\D)  \\
\Leftrightarrow &  \langle J(\Phi'a),g \rangle=0\,\, \forall \,\,  g \in H^2_\clf(\D) \\
\Leftrightarrow & J (\Phi'a) \in [H^2_{\clf}(\D)]^\perp
\end{align*}
which is not possible unless $\Phi = 0$.

On the other hand, let $S_2: H^2_{\clf}(\D) \rightarrow H^2_{\E}(\D)$ such that $M_z^\cle S_2=S_2(M_z^{\clf})^*$. Hence, for each $n \geq 0$, 
\[
(M_z^\cle)^n S_2=S_2(M_z^{\clf})^{*n}
\]
which implies $S_2=(M_z^\cle)^{*n}S_1(M_z^\clf)^{*n}$. Since $M_z^{\clf}$ is pure, we have $S_2 = 0$. 

Hence, $\clm$ is hyperinvariant subspace for $X$ if and only if $\clm$ is one of the following forms:
\begin{enumerate}
\item $\{0\} \oplus [vH^2_{\clf}(\D)]^\perp$, 
\item $\{0\} \oplus H^2_\clf(\D)$, or
\item $uH^2_\cle(\D) \oplus H^2_\clf(\D)$,
\end{enumerate}
where $u$ and $v$ are scalar-valued inner functions.
\end{proof}

\begin{thm}
The only hyperinvariant subspaces of $(M_z^\clf)^* \oplus  J_k$ on $H^2_\clf(\D) \oplus  \mathbb{E}_k$   are $ [vH^2_{\clf}(\D)]^\perp \oplus \cln (J_k^{n_k})$ and $ H^2_{\clf}(\D) \oplus  \cln (J_k^{n_k})$, where $n_k \in \{0,1,\ldots,k\}$ and $v$ is any scalar-valued inner function such that $v \in z^{n_k}H^2(\D)$.
\end{thm}

\begin{proof}
Set
\[
W=(M_z ^{\clf})^*\oplus J_k  \mbox{~~on~~} H^2_{\clf}(\D) \oplus \mathbb{E}_k.
\]
Again $\mathbb{E}_k$ is identified as a subspace of $H^2_{\cle_k}(\D)$, namely, $[z^kH^2_{\cle_k}(\D)]^\perp$. 
Suppose $\clm$ is hyperinvariant for $W$. We shall again make use of  Lemma \ref{thm-hypinv-directsum} in what follows. Then $\clm$ is either of type 
\[
\clm = H^2_{\clf}(\D) \oplus  \cln (J_k^{n_k })\mbox{~~or~~} \clm=[v H^2_{\clf}(\D)]^\perp \oplus \cln (J_k^{n_k }),
\]
where $v\in H^\infty$ is inner and $n_k  \in \{0,1,2,\ldots,k\}$.

Let $S_1:  H^2_{\clf}(\D) \rightarrow \mathbb{E}_k$ such that $J_kS_1=S_1(M_z^\clf)^*$. Then
\[
S_1(M_z^\clf)^{*k}=J_k^kS_1=0.
\]
Thus $S_1 = 0$. 

Now let $S_2: \mathbb{E}_k \rightarrow  H^2_{\clf}(\D)$ such that 
\begin{equation}\label{intertwin}
(M_z^\clf)^*S_2=S_2J_k.
\end{equation}
Since $M_z^{\cle_k}$ on $H^2_{\cle_k}(\D)$ is the minimal isometric dilation of $J_k$, 
\[
J_kP_{\mathbb{E}_k}=P_{\mathbb{E}_k}M_z^{\cle_k},
\]
where $P_{\mathbb{E}_k}$  is  the orthogonal projection from $H^2_{\cle_k}(\D)$ onto $\mathbb{E}_k$.  Then 
\[
(M_z^\clf)^*S_2P_{\mathbb{E}_k}=S_2J_kP_{\mathbb{E}_k}=S_2P_{\mathbb{E}_k}M_z^{\cle_k}.
\]
Set $Y=S_2P_{\mathbb{E}_k}$. Then $(M_z^\clf)^*Y=YM_z^{\cle_k}$ implies 
\[
Y=H_\Phi \mbox{~~for some~~} \Phi \in L^\infty_{\mathcal{B}(\cle_k,\clf)},
\]
and hence
\[
S_2=Y|_{\mathbb{E}_k}=H_\Phi |_{\mathbb{E}_k}.
\]

Now let  $\Phi \in L^\infty_{\clb(\cle_k,\clf)}$  such that $(M_z^\clf)^*H_\Phi|_{\mathbb{E}_k}=H_\Phi J_k$. Then for $a \in \mathcal{E}_k$ and $m \geq 0$, we have
\[
	H_\Phi(ae^{i(k+m)t})=(M_z^\clf)^{*m+1}H_\Phi(ae^{i(k-1)t})=H_\Phi J_k^{m+1}(ae^{i(k-1)t})=0.	
\]

Thus $z^kH^2_{\cle_k}(\D) \subseteq \cln(H_\Phi)$. The converse is also true. Hence (\ref{intertwin}) holds if and only if $S_2=H_\Phi|_{\mathbb{E}_k}$, where $\Phi \in L^\infty_{\clb(\cle_k,\clf)}$  such that   $z^kH^2_{\cle_k}(\D) \subseteq \cln(H_\Phi)$.

Now for $\clm=   H^2_{\clf}(\D) \oplus  \cln (J_k^{n_k })$, $S_2( \cln (J_k^{n_k })) \subseteq  H^2_{\clf}(\D)$ is obvious. 

For $\clm=[v H^2_{\clf}(\D)]^\perp \oplus  \cln (J_k^{n_k })$ being hyperinvariant for $W$,  we  have 
\[
H_\Phi(\cln (J_k^{n_k })) \subseteq [v H^2_{\clf}(\D)]^\perp, 
\]
for all $\Phi \in L^\infty_{\clb(\cle_k,\clf)}$  such that   $z^kH^2_{\cle_k}(\D) \subseteq \cln(H_\Phi)$. 

For $n_k=0$, it is trivially true. Let $n_k \geq 1$.
Note that for any $\phi_k \in \clb(\cle_k,\clf)$, we can define $\Phi_k \in L^\infty_{\clb(\cle_k,\clf)}$ by
\begin{equation}\label{Phi}
		\Phi_k(e^{it})= \phi_ke^{-i(k-1)t}\quad\quad (\text{a.e.\  on } \T).
\end{equation}
Then, for $a \in \cle_k$, 
\[
H_{\Phi_k}(ae^{ikt})=P_+^\clf J(\Phi_k ae^{ikt})=P_+^\clf (\phi_k(a)e^{-it})=0.
\]
And hence $z^kH^2_{\cle_k}(\D) \subseteq \cln(H_{\Phi_k})$.  It follows that
\[
		H_{\Phi_k}(\cln(J_k^{n_k})) \subseteq [vH^2_\clf(\D)]^\perp.
\]
Choose $\phi_k\in \clb(\cle_k,\clf)$ and $a \in \cle_k$ such that $\phi_k(a)\neq 0$. Note that $az^{k-n_k} \in \cln(J_k^{n_k})$. Suppose $v =\sum\limits_{m=0}^\infty v_mz^m$ $(v_m \in \C)$. Then, for $\Phi_k$ as defined in (\ref{Phi}), 
\begin{align*}
			& \langle H_{\Phi_k}(ae^{i(k-n_k)t}), vg\rangle =0 \,\,\forall \,\,g \in H^2_\clf(\D)\\
	\implies	& \langle P_+^\clf J(\Phi_k ae^{i(k-n_k)t}), vg \rangle  =0  \,\, \forall\,\, g \in H^2_\clf(\D)\\
	\implies	& \langle J(\Phi_k ae^{i(k-n_k)t}), vg \rangle=0 \,\, \forall \,\, g \in H^2_\clf(\D)\\
	\implies 	&\frac{1}{2\pi}\int_{0}^{2\pi}\left\langle \phi_k(a)e^{i(n_k-1)t}, v(e^{it})g(e^{it}) \right\rangle_\clf=0 \,\, \forall \,\, g \in H^2_\clf(\D)  \\
	\implies 	&\frac{1}{2\pi}\int_{0}^{2\pi}\left\langle \overline{v(e^{it})}\phi_k(a)e^{i(n_k-1)t},g(e^{it}) \right\rangle_\clf=0 \,\, \forall \,\, g \in H^2_\clf(\D)  \\
	\implies 	& \langle T_\Psi a, g\rangle =0\,\, \forall \,\, g \in H^2_\clf(\D),
\end{align*}
where $\Psi=\bar{v}\phi_kz^{n_k-1}=\sum\limits_{m=0}^\infty \overline{v_m}\phi_kz^{n_k-m-1}$. Hence
\[
			T_\Psi a=0.
\]

Thus $\overline{v_m}\phi_k(a)=0$ for all $m \leq n_k-1$. But since $\phi_k(a)\neq 0$, we get $v_m=0$ for all $m \leq n_k-1$, i.e., $v \in z^{n_k}H^2(\D)$. 

Conversely, for  $v \in z^{n_k}H^2(\D)$ and $\Phi \in L^\infty_{\clb(\cle_k,\clf)}$ such that   $(M_z^\clf)^*H_\Phi|_{\mathbb{E}_k}=H_\Phi J_k$, it is trivial to check that $H_\Phi(\cln(J_k^{n_k}))\subseteq [vH^2(\D)]^\perp$. Hence $\clm$ is hyperinvariant for $W$.
Therefore, $\clm$ is hperinvariant for $W$ if and only if it is either of the form
\[
\clm= H^2_{\clf}(\D) \oplus \cln (J_k^{n_k}) \mbox{~~~~or~~~~} \clm=[vH^2_\clf(\D)]^\perp \oplus \cln(J_k^{n_k}),
\]
where $v$ is a scalar-valued inner function such that $v \in z^{n_k}H^2(\D)$ and $n_k \in \{0,1,\ldots k\}$.
\end{proof}

From the above deliberation, we obtain the complete characterization of hyperinvariant subspaces of a c.n.u.\ power partial isometry which is stated as follows.

\begin{thm}
A closed subspace $\clm$ of $\he \oplus H^2_{\clf}(\D) \oplus \bigoplus\limits_{k=1}^\infty (\cle_k \otimes \C^k)$ is hyperinvariant for  a  c.n.u.\ power partial isometry $ T=M_z \oplus M_z^* \oplus \left(\bigoplus\limits_{k=1}^\infty J_k\right)$ if and only if $\clm$ is one of the following forms:
\[
\clm =\begin{cases}
\{ 0 \} \oplus H^2_{\clf}(\D) \oplus \left(  \bigoplus\limits_{k=1}^\infty \cln (J_k^{n_k })\right)\\
\{0 \}  \oplus  [vH^2_{\clf}(\D)]^\perp \oplus \left(\bigoplus\limits_{k=1}^\infty \cln (J_k^{n_k })\right)\\
    u\he  \oplus  H^2_{\clf}(\D)  \oplus \left(  \bigoplus\limits_{k=1}^\infty \cln (J_k^{n_k })\right)
\end{cases},
\]
where $n_k  \in \{0,1,\dots,k\}$ with $n_k  \leq n_{k+1} \leq n_k +1$ and $u,v \in H^\infty$ are inner functions such that $u \in z^{k-n_k}H^2(\D)$ and $v \in z^{n_k}H^2(\D)$ for all $k \geq 1$.
\end{thm}

\vspace{0.2 cm}

\label{Ref}

\end{document}